\documentclass[final,leqno]{siamltex}

\usepackage{amscd}
\usepackage{amssymb}
\usepackage{amsmath}
\usepackage[dvips]{graphicx}
\usepackage{amsfonts}
\usepackage{amsopn}
\usepackage{cases, comment}
\usepackage{hyperref}

\numberwithin{equation}{section}
\newcounter{mtheorem}

\newtheorem{thm}{Theorem}[section]

\newtheorem{prop}[thm]{Proposition}

\newtheorem{rmk}[thm]{Remark}

\def\Reals{\mathbb{R}}

\def\R2d{{\Reals}^{2d}}

\def\calA{{\mathcal A}}

\def\calH{{\mathcal H}}

\def\calL{{\mathcal L}}
\def\calM{{\mathcal M}}
\def\calPac_1on{{\mathcal M}on}
\def\calN{{\mathcal N}}

\def\calP{{\mathcal P}}

\def\calU{{\mathcal U}}

\def\calZ{{\mathcal Z}}

\def\calPac{{\mathcal P}^{ac}}

\def\x{{\bf x}}
\def\y{{\bf y}}

\def\id{{\bf id}}

\def\div{{\rm div}}
\def\dom{{\rm dom}}

\def\endproof{\ \hfill $\square$ \bigskip} \def\proof{\noindent {\bf Proof:}}
\def\proof#1{\noindent {\bf Proof{#1}:}}

\def\vrho{\varrho}

\def\calPac{{\mathcal P}^{ac}}

\def\Arg{{\it Argmin} \Pi}

\def\x{{\bf x}}
\def\y{{\bf y}}

\def\id{{\bf id}}

\def\div{{\rm div}}
\def\dom{{\rm dom}} 
\def\spt{{\rm spt}}

\title{On a model of forced axisymmetric flows.}

\author{Mike Cullen\thanks{Met Office, Exeter, EX1 3PB, UK Exeter({\tt mike.cullen@metoffice.gov.uk}). MC's contribution is Crown Copyright.}
        \and Marc Sedjro\thanks{Rwth Aachen University, Aachen, Germany({\tt sedjro@instmath.rwth-aachen.de}).} The material in this paper was part of MS's dissertation.}

\begin{document}

\maketitle

\begin{abstract}
In this work, we consider a model of forced axisymmetric flows which is derived from the inviscid Boussinesq equations. What makes these equations unusual is the boundary conditions they are expected to satisfy and the fact that the boundary is part of the unknowns.  We show that these flows give rise to an unusual Monge-Ampere equations for which we prove the existence and the uniqueness of a variational solution. We take advantage of these Monge-Ampere equations and construct a solution to the model.
\end{abstract}



\pagestyle{myheadings}
\thispagestyle{plain}
\markboth{MIKE CULLEN AND MARC SEDJRO}{ON A MODEL OF FORCED AXISYMMETRIC FLOWS.}


\section{Introduction}
\label{}
In this paper, we consider a model of forced axisymmetric flows in the absence of viscosity. This model was introduced by \cite{Eliassen} to study  the structure of tropical cyclones. The solution can be regarded as an axisymmetric vortex which is stable to axisymmetric perturbations and which evolves slowly in time under the action of forcing, \cite{EliassenK} \cite{Fjortoft}. Shutts {\em et al.} \cite{Shutts} developed a discrete procedure for solving this problem within a rigid axisymmetric boundary. We extend this procedure to the continuous case by using mass transportation methods, as reviewed by \cite{Cullen2006}. We also propose and solve a novel free boundary version which is more physically appropriate, as it allows the vortex to evolve within an ambient fluid at rest. Mass transportation methods have been applied successfully to a free boundary problem by \cite{CullenGangbo}, but our problem differs in important respects from theirs. The time dependent domain where the fluid evolves, in the cylindrical coordinates $(\lambda,r,z)$,  is of the form 
\begin{equation}\label{eq: boundary cond 0}
\Gamma_{\zeta_t}=\left\lbrace (\lambda, r,z) :  0\leq \lambda\leq 2\pi, \;\; 0\leq z\leq H,\;\;  r_0\leq r\leq \zeta(t,\lambda,z) \right\rbrace,
\end{equation}

\noindent where the boundary $r=\zeta_t$ is a material surface and $r_0, \; H$ are positive real numbers. We have used the notation $S_t= S(t,\cdot,\cdot)$. The temperature $\theta$ within the domain of the vortex (where the PDEs are considered)  is higher  than the temperature  in the ambient fluid  which is maintained at a constant temperature $\theta_0$  in a rotating framework where the Coriolis coefficient is  $\Omega >0$. We denote by ${\bf u}=(u,v,w)$ the velocity of the fluid in cylindrical coordinates. The material derivative associated to this velocity in cylindrical coordinates takes the form $\frac{D}{Dt}:=\frac{\partial}{\partial t}+\frac{u}{r}\frac{\partial}{\partial \lambda} +v\frac{\partial}{\partial t}+w\frac{\partial}{\partial z}$. The pressure inside the vortex is denoted by $\varphi.$ \\

We  follow procedures proposed by  Craig \cite{Craig} to solve the time evolution of the vortex under axisymmetric forcing.
The unknowns of the problem are ${\bf u}=(u,v,w), \theta, \varphi, \zeta$. We start by writing the equations for forced almost axisymmetric flows, as derived by Craig \cite{Craig}: 
\begin{equation}{\label{eq:craig}}
\begin{cases}
\frac{{\rm D} u}{{\rm D} t}+\frac{uv}{r}+2\Omega v +\frac{1}{r}\frac{\partial\varphi}{\partial \lambda}&=\frac{1}{r}F(t,\lambda,r,z),\\
\frac{{\rm D} \theta}{{\rm D} t}& =S(t,\lambda,r,z),\\
\frac{u^2}{r}+2\Omega u &=\frac{\partial\varphi}{\partial r}, \\
\frac{\partial\varphi}{\partial z}-g\frac{\theta}{\theta_0} &=0, \\
\frac{1}{r}\frac{\partial}{\partial \lambda}(u)+ \frac{1}{r}\frac{\partial}{\partial r}(rv)+\frac{\partial w}{\partial z} &=0.
\end{cases}
\end{equation}
We consider these equations with Neumann conditions imposed  on the rigid boundary  while a kinematic boundary condition is imposed on the free boundary and the pressure is required  at each time $t$ to vanish at $\{r=\zeta^t\}\cap \{z=H\}  $. In Craig \cite{Craig}, the free boundary condition was replaced by a decay condition as $r\rightarrow\infty$. $F(t,\lambda,r,z)$ and $S(t,\lambda,r,z)$ are prescribed forcing terms of the system.

Despite the fact that  the equations for almost axisymmetric flows (when $F=S=0$ in (\ref{eq:craig})) are derived as   approximations to the inviscid Boussinesq equations, surprisingly, they have retained  quite the same level of complexity and formally are known to have kept the infinite dimensional Hamiltonian structure already present in the inviscid  Boussinesq equations.  From a physical perspective, we  are interested in solutions that are stable in the sense that they correspond to a minimum energy state with respect to parcel displacements that preserve the angular momentum and the potential temperature (see \cite{EliassenK}). As suggested in \cite{Sedjro}, one of the main obstructions we run into while implementing  the solution procedure we propose for solving the almost axisymmetric flow equations comes from our inability to find adequate  regularity properties for the pressure  $\varphi$ with respect to $\lambda$ as the system evolves in time. In this paper, we set aside this difficulty by considering  the equations for forced axisymmetric flows. The forcing terms can be considered as representing the effects of the non-axisymmetric parts of the real flow on the axisymmetric part. The solution, as we will show, shares the same stability property as the full  system of equations and thus sheds some light on the structure of almost axisymmetric flows.

\subsection{The Axisymmetric Model}
We assume that the  quantities and operators involved in (\ref{eq:craig}) do not depend on $\lambda$   (in particular, here ${D  \over Dt}= {\partial \over \partial t}+  v{\partial \over \partial r}+  w{\partial \over \partial z}$ ) and thus obtain the $2$-dimensional system of equations:

\begin{subnumcases}{\label{eq:non physical}}
\frac{{\rm D}  u}{{\rm D} t}+\frac{ u v}{r}+2\Omega  v &$=\frac{1}{r} F(t,r,z),\label{first1}$\\
\frac{{\rm D} \theta}{{\rm D} t} &$= S(t,r,z),\label{second1}$\\
\frac{ u^2}{r}+2\Omega  u &$=\frac{\partial\varphi}{\partial r},\label{third1}$ \\
\frac{\partial \varphi}{\partial z}-g\frac{\theta}{\theta_0} &=$0, \label{fifth1}$\\
 \frac{1}{r}\frac{\partial}{\partial r}(r v)+\frac{\partial  w}{\partial z} &=$0$.\label{fourth1}  
\end{subnumcases}

 The  above equations are to be solved in the moving  domain 
\begin{equation}\label{eq: boundary cond 1}
\Gamma_{\zeta_t}=\left\lbrace ( r,z) : \;\; 0\leq z\leq H,\;\;  r_0\leq r\leq \zeta(t,z) \right\rbrace,
\end{equation}
 where $\zeta$ is a free boundary. The conditions  on the boundary are given by 
\begin{equation}\label{eq: non physical free boundary}
\begin{cases}
 \langle(  v_t, w_t) , {\bf n_t}\rangle=0 \qquad \qquad \text{ on } \{r=r_0\}\cup\{z=0\}\cup \{z=H\},\\
 \frac{\partial \zeta_t}{\partial t} + w\frac{\partial \zeta_t}{\partial z} \;\;\;=  v  \qquad \qquad \text{on}\;\; \{ r=\zeta(t,z)\},
\end{cases}
\end{equation} 
along with
\begin{equation}\label{eq: bdary cond on pressure}
 \varphi(t,\zeta(t, H),H)=0.
\end{equation}

Here, ${\bf n_t}$  is the unit outward normal vector field at time $t$ of the boundary  $\partial\Gamma_{\varsigma_t}$. $ F(t,r,z)$ and $ S(t,r,z)$ are prescribed functions. The Hamiltonian relevant to the system  (\ref{eq:non physical}) is given by
\begin{equation}\label{eq: Hamiltonian}
 \int_{\Gamma_{\zeta}}\left( \frac{u^2}{2}-\frac{g\theta }{\theta_0}z\right)rdrdz.
\end{equation}

 In order to obtain stable solutions as  discussed above, we require the pressure to satisfy the following stability condition:

\begin{equation}\label{eq: stability condition}
\nabla_{r,z}\left( \varphi + \Omega^2\frac{r^2}{2} \right)\quad \text{  is    \textit{invertible}.}
\end{equation}

These equations are supplemented by the initial conditions
$$( u, v, w)_{|t=0}=( u_0, v_0, w_0), \; \theta_{|t=0}= \theta_0,\; \varphi_{|t=0}=\varphi_0, \; \zeta_{|t=0}=\zeta_0$$
that are required to satisfy (\ref{third1})-(\ref{fourth1}), the first equation in (\ref{eq: non physical free boundary}), (\ref{eq: bdary cond on pressure}) and (\ref{eq: stability condition}). 

\subsection{Continuity equation corresponding to the 2D Axisymmetric Flows with Forcing Terms}

\label{sec: Change of variables into the Dual space  and Justification}
Interestingly, given enough regularity, the equations of stable forced axisymmetric flows can be reformulated as a continuity equation for some family of probability measures $\left\lbrace\sigma_t\right\rbrace_{t\in[0,T]}$ ($T>0$) and a some  velocity field $V$ in a set of transformed variables:
 \begin{equation}\label{eq00: Continuity equation unphysical problem}
 \begin{cases}
  \frac{\partial\sigma}{\partial t}+ \div (\sigma V_t)=0 \qquad (0,T)\times \mathbb{R}^2,\\
 \sigma|_{t=0}=\sigma_0
 \end{cases}
 \end{equation}
  in the sense of distributions. Indeed, let $\varphi :(0,\infty)\times (0,H)\times (r_0, \infty)\longrightarrow \mathbb{R}$  be smooth and use the change of variable $ 2s= r_0^{-2}-r^{-2}$ to define 
\begin{equation}\label{eq:P in term of the pressure}
 P_t(s,z)=\varphi_t(r,z)+\frac{\Omega^2r^2}{2}.
\end{equation}
 We denote by $\Psi_t:=\Psi_t(\Upsilon,Z)$ the Legendre transform of $P_t$ for each $t$ fixed.
Let $\zeta$ be smooth such that $r\chi_{\Gamma_{\zeta_t}(r,z)}$ is a probability density function for each $t$ fixed. The change of variable $2s= r_0^{-2}-r^{-2}  $ induces   $\rho_t:[0,H]\longmapsto[0;1/(2r_0^2)) $ such that
\begin{equation}\label{eq: expressions of measure in physical space}
r\chi_{\Gamma_{\zeta_t}}(r,z)drdz= e(s)\chi_{D_{\rho_t}}(s,z)dsdz,
\end{equation}
with
$$ D_{\rho_t}=\left\lbrace (s,z): 0\leq s\leq \rho_t(z), z\in [0,H] \right\rbrace \qquad \text{ and } \qquad e(s)=r_0^4/(1-2sr_0^2)^2 \quad \text{for}\quad 0\leq 2r_0^2s < 1.$$

 Assume that $\varphi_t$ is such that $P_t$ is convex and $\varphi$ satisfies (\ref{eq: stability condition}) and define a family of probability measures $\left\lbrace\sigma_t\right\rbrace_{t\in[0,T]}$, absolutely continuous with respect to Lebesgue, by
 
  \begin{equation}\label{eq intro: Monge Ampere with parameter} 
 e(\partial_{\Upsilon} \Psi)\det(\nabla_{\Upsilon,Z}^{2}\Psi)= \sigma, \qquad
 \nabla_{\Upsilon,Z}\Psi( spt(\sigma ))= D_{\rho}. 
  \end{equation} 
 
 If equations  (\ref{eq:non physical})-(\ref{eq: bdary cond on pressure}) admit a solution 
 $\zeta,\varphi$,  then $t\longmapsto\sigma_t$ is an absolutely continuous curve in the space of Borel probability measures equipped with the Wasserstein distance and satisfies (\ref{eq00: Continuity equation unphysical problem}) with
 \begin{equation}\label{eq1 unphysical1: velocity in dual space} 
  V_t= {\textstyle \left(\;2\sqrt{\Upsilon} F_t\left( \frac{r_0}{\sqrt{1-2r_0^2\frac{\partial \Psi}{\partial \Upsilon}}},\frac{\partial\Psi}{\partial Z} \right)  ,\; \frac{g}{\theta_0}  S_t\left( \frac{r_0}{\sqrt{1-2r_0^2\frac{\partial \Psi}{\partial \Upsilon}}},\frac{\partial\Psi}{\partial Z} \right)  \right). }
 \end{equation}
 Conversely,  assume that $(P,\Psi, \sigma)$ is a smooth solution of (\ref{eq00: Continuity equation unphysical problem}) and (\ref{eq1 unphysical1: velocity in dual space}) such that    $(P_t,\Psi_t)$ are Legendre transforms of each other  and $\nabla P_t$, $\nabla \Psi_t$ are inverse of each other for  each $t$ fixed. If, in addition, there exists $\rho$ smooth such that (\ref{eq intro: Monge Ampere with parameter}) holds  then we can construct a solution for the forced axisymmetric flows as shown in section \ref{sec: Derivation of the continuity equation}. We note that the equations (\ref{eq intro: Monge Ampere with parameter})  serve as a change of variable between (\ref{eq00: Continuity equation unphysical problem}), (\ref{eq1 unphysical1: velocity in dual space}) on the one hand, and (\ref{eq:non physical})-(\ref{eq: non physical free boundary})  on the other hand.\\
 
 It is important to emphasize that the solution we construct for (\ref{eq:Monge Ampere}) in this paper is not smooth enough  so as to make the transition from (\ref{eq00: Continuity equation unphysical problem}) and (\ref{eq1 unphysical1: velocity in dual space}) to (\ref{eq:non physical})-(\ref{eq: bdary cond on pressure}).

As shown in (\cite{Sedjro}), the Hamiltonian given  in (\ref{eq: Hamiltonian}) for  stable forced axisymmetric flows can be expressed, via  appropriate changes of variables, by  the following functional :
  
  \begin{equation}\label{eq intro:primal1} 
\mathcal{P}(\mathbb{R}^2) \ni \sigma\longmapsto H_{\ast }(\sigma):=
  \inf_{(\rho, \gamma)\in\mathfrak{L}_{\sigma}}\int_{D_\rho \times \Delta} \left( -\langle p, q \rangle +{\Upsilon \over 2 r_0^2}-\Omega \sqrt \Upsilon + {r_0^2 \Omega^2 \over 2(1-2 r_0^2 s)}\right)  \gamma(dp, dq)  
  \end{equation}
with $p=(s,z)$, $q=(\Upsilon,Z)$ and $$\mathfrak{L}_{\sigma}=\Bigl\{(\rho,\gamma): \rho\in\calH_0  , \int_{D_{\rho}}e(s)dsdz=1, \gamma \in \Gamma(e(s)\chi_{D_{\rho}}\calL^2, \sigma)\Bigr\}.$$

   Here, $\calH_0$  consists of all Borel functions  $\rho : [0,H]\longmapsto [0,1/(2r_0^2)$. The minimization problem in (\ref{eq intro:primal1}) has a dual formulation

  \begin{equation}\label{eq intro:dual1} 
    H_{\ast}(\sigma)=  \sup \left(  \int_{\mathbb{R}^2} \Bigl({\Upsilon \over 2r_0^2} -\Omega \sqrt \Upsilon- \Psi\Bigr) \sigma(dq)+\inf_{\rho \in \calH_0} \int_0^H \int_0^{\rho(z)} \Bigl( \dfrac{\Omega^2 r_{0}^{2}}{2(1-2s r_{0}^{2})}-P(s,z) \Bigr) e(s)ds\right). 
  \end{equation}

  The supremum in (\ref{eq intro:dual1}) is taken over the set 
  \begin{equation}\label{eq intro:for calU}
   \mathcal{U}:=\Bigl\{ (\Psi, P)\in C(\Reals^2_+)\times C(\bar {\mathcal{D}}): P(p)+\Psi(q) \geq  \langle p, q \rangle \text{ for all } (p,q)\in \mathcal{D}\times\mathbb{R}_+^2 \Bigr\},
  \end{equation}
  
where the set $\mathcal{D}$  will be given in (\ref{eq: define D}). It turns out that if $ \rho^{\sigma}$ is a minimizer in (\ref{eq intro:primal1}) and $ (P^{\sigma},\Psi^{\sigma})$, Legendre transforms of each other, is a maximizer of \eqref{eq intro:dual1}  then $( \rho^{\sigma},P^{\sigma}, \Psi^{\sigma})$ solves (\ref{eq:Monge Ampere}) and moreover, we have  $\nabla \Psi^{\sigma}\circ\nabla P^{\sigma}= \id \quad  e(s)\chi_{D_{\rho_0}}(s,z)\calL^2-\text{almost everywhere}$ and $\nabla P^{\sigma}\circ\nabla \Psi^{\sigma}= \id \quad \sigma-\text{almost everywhere} $.

\subsection{ Challenges and Plan of the paper}

We show that (\ref{eq:Monge Ampere}) admits a unique variational solution $(P^{\sigma}, \Psi^{\sigma},\rho^{\sigma})$ in the sense of (\ref{eq intro:primal1}) and  (\ref{eq intro: Monge Ampere with parameter})
  which we exploit to get that the operator $\sigma\mapsto V_t[\sigma]$ is continuous, leading to the existence of a global solution in (\ref{eq00: Continuity equation unphysical problem}). The difficulty in obtaining the existence of a minimizer in (\ref{eq intro:primal1} ) lies in the fact that the set of functions $\left\lbrace\chi_{D_{\rho}} (s,z)\right\rbrace_{\rho\in\calH_0} $ is not closed in the $L^{\infty}$ weak* topology.  This  is an obstacle we bypass  by observing that
$$\bar I[\sigma](\rho^{\#})\leq \bar I[\sigma](\rho)$$
 if $\rho^{\#}$ is a monotone rearrangement of $\rho$ (see \cite{Sedjro}). The existence follows easily from the fact that the monotone  functions are precompact with respect to pointwise topology. But the uniqueness proved extremely challenging in the sense that we do not know any strict convexity property for the functional with respect to any interpolation  we could think of. In section \ref{sec:Duality Methods and Monge-Ampere Problem},  we resort to a duality argument and discover a twist condition for a certain functional  which ensures uniqueness in (\ref{eq intro: Monge Ampere with parameter}) and furthermore, we show that the boundary of the domain  $D_{\rho^{\sigma}}$ is a finite union of graphs of Lipschitz functions. Before that, we fix the notations and give some definitions in section \ref{sec: notation}. In section \ref{sec: Derivation of the continuity equation}, we derive the continuity equation corresponding to the forced axisymmetric flows. We state some stability results in section \ref{sec:Some stability results} which are used to obtain the continuity of the operator $\sigma\mapsto V_t[\sigma]$ and thus the existence of  a global solution in time  to the continuity equation 
 (\ref{eq00: Continuity equation unphysical problem})-(\ref{eq1 unphysical1: velocity in dual space}) in section \ref{sec:Continuity equation for the Toy Model}. We note that the construction of a global solution follows a scheme pioneered by Ambrosio-Gangbo \cite{AmbrosioGangbo}.

\section{Notation and Definitions}
\label{sec: notation}
 In this section we introduce some notations and recall some standard definitions. Let $d\in\mathbb{N}$.
\begin{itemize}
\item  For any real number $x$, $[x]$ denotes the integer part of $x$.
\item For $A\subset \mathbb{R}^d$, $\bar A$ is the closure of $A$.
\item If $f: A\subset \mathbb{R}^d \longmapsto \mathbb{R}$ is a Lipschitz continuous then $\displaystyle{Lip(f)\equiv\sup_{\substack{\x,\y\in A\\ \x\neq\y}}\frac{|f(\x)-f(\y)|}{|\x-\y|}}$ denotes the lipschitz constant of $f$.
\item Let $\mathbb{X}\subset \mathbb{R}^d  $ be a convex set. If $P: \mathbb{X} \longmapsto \mathbb{R}$ is a convex function and $p\in \mathbb{X}$, the subdifferential of P at $p$, denoted by $\partial_{.} P(p)$,  is defined as $\partial_{.} P(p)=\left\lbrace s\in \mathbb{R}^d: P(q)\geq P(p) + \langle s,q-p\rangle \;\forall q\in\mathbb{X} \right\rbrace $.
  
\item  $\mathcal{P}(\mathbb{R}^d)$ is the set of all Borel probability measures on $\mathbb{R}^d$ and   $\mathcal{P}_p(\mathbb{R}^d)$ ($1\leq p<\infty$) is the set of all probability measures with $p-$ finite moments.  If $R>0$ then  $\mathcal{P}_{[R]}$ will denote the subset of $\mathcal{P}(\mathbb{R}^2)$ consisting of borel  probability measures supported in $[0,R]^2$. For $\sigma\in\mathcal{P}(\mathbb{R}^2)$, $\spt(\sigma)$ will denote the support of $\sigma$.

  \item Given $\mu_0$ and $\mu_1 \in \mathcal{P}(\mathbb{R}^d)$, we denote by $\Gamma(\mu_0,\mu_1)$ the set of all Borel measures on $\mathbb{R}^d\times\mathbb{R}^d$  whose first and second marginal are respectively $\mu_0$ and $\mu_1$. We say that a Borel map $T$ \textit{pushes} $\mu_0$ \textit{ forward} to  $\mu_1$   and write $T_{\#}\mu_0=\mu_1$ if $\mu_1(A)=\mu_0(T^{-1}(A))$ for any $A\subset\mathbb{R}^d$ borel. \\
  If, in addition, $\mu_0$ $\mu_1$ are of $p-$ finite moments then the ($p$-th) Wasserstein distance between the Borel  measures $\mu_0$ and $\mu_1$  is defined by
   \begin{equation}\label{def of Wasserstein distance}
    W^{p}_{p}(\mu_0,\mu_1)=     \min\left\lbrace \int_{\mathbb{R}^d\times\mathbb{R}^d}|x-y|^pd\gamma: \gamma\in \Gamma(\mu_0,\mu_1)\right\rbrace. 
   \end{equation}
  The set of minimizers in (\ref{def of Wasserstein distance}) is denoted $\Gamma_0(\mu_0,\mu_1)$.

 \item Let $1\leq m\leq \infty$ and $a,b$ reals numbers such that $a<b$. A curve $ \mu :(a,b)\longrightarrow \mathcal{P}_p(\mathbb{R}^d)$ is said to belong to  $AC_m\left( (a,b); \mathcal{P}_{p}(\mathbb{R}^d) \right)$  if there exists $ {\bf m}\in L^m(a,b)$ such that 
 $$W_p(\mu_t,\mu_s)\leq \int_t^s {\bf m}(r)dr \qquad  \text{for all } a<t\leq s<b.$$
 
 Curves in  $AC_m\left( (a,b); \mathcal{P}_p(\mathbb{R}^d) \right)$  are said to be $m-$absolutely continuous.

\item If $ X=(X_1,X_2)$ is a vector field then $\frac{D^{X}}{Dt}:= \frac{\partial}{\partial t} + X_1\frac{\partial}{\partial r}+X_2\frac{\partial}{\partial z}$. We simply write $\frac{D}{Dt}$ whenever the velocity field is understood from the context.

  \item Throughout this manuscript, $R_0, r_0, H$ are positive constants and we set
  
  \begin{equation}\label{eq: define D} 
   \mathcal{D}:=[0,1/2r_0^2)\times [0,H].
  \end{equation}
  \item  $\Delta$ is a subset of $[0,R_0]^2$. We assume that  $\Delta$ is closed throughout the paper.
\end{itemize} 



\section{Derivation of the continuity equation}
\label{sec: Derivation of the continuity equation}
In this section, we discuss the equivalence of  Forced axisymmetric flows with a certain class of continuity equations when enough regularity is assumed.

We consider the following equations where the unknowns are $\textbf{u}=(v,w)$ and $\zeta$ :

\begin{equation}\label{eq1: conservation of mass physical} 
 \begin{cases}
  \frac{1}{r}\partial_{r}(rv)+\partial_{z} w=0 \;\; on \;\; \Gamma_{\zeta_t},\\
\langle{{\bf u}}_t , {\bf n}_t \rangle=0 \;\;  \{z=0 \}\cup\{z=H \}\cup \{r=r_0\},\\
 \frac{\partial \zeta_t}{\partial t} + w\frac{\partial \zeta_t}{\partial z}=v   \qquad on\;\; \{ r= \zeta(t,z)\}.\\
 \end{cases}
\end{equation}
Here, ${\bf n}_t$ is outward unit normal vector of the boundary  $\partial\Gamma_{\varsigma_t}$ for each $t$ fixed. We point out that 
 the equations in (\ref{eq1: conservation of mass physical}) express the conservation of the mass in the physical space for the axisymmetric flows.\\

The complete  proof of the following Lemma in dimension 3 can be found in \cite{Sedjro} and can be easily adapted to a  2-dimensional case.\\

\begin{lemma}\label{le: conservation of total mass} 
Let $\zeta$ be smooth and $\textbf{u}=(v,w)$ a smooth velocity field. Let  $\frak F$ be a smooth  function on $(0, T)\times\mathbb{R}^2$  such that $\frak F_t$ invertible for each $t\in[0, T]$ and the inverse function is smooth.\\
 Let $  \sigma :(0,T)\longmapsto \mathcal{P}(\mathbb{R}^2)$ and $V$ be such that
\begin{equation}\label{eq1 : definition of sigma and condition on F} 
 \frak F_{t\#}(r\chi_{\Gamma_{\zeta_t}}(r,z)\mathcal{L}^2)=\sigma_t \quad \text{for each t fixed}\qquad\hbox{and}\qquad V\circ\frak F_t= \frac{D^{\mathbf{u}} \frak F_t}{D t}.
\end{equation}
 
  If $v,w,\zeta$  solve (\ref{eq1: conservation of mass physical}) then $\sigma$ is an absolutely continuous curve, $V$ is smooth and $\sigma, V$ solve 
 
 \begin{equation}\label{eq1: conservation of mass dual space} 
   \frac{\partial \sigma_t}{\partial t}+\div(V_t\sigma_t)=0 \qquad (0,T)\times \mathbb{R}^3.
\end{equation}
in the sense of distributions. Conversely, assume that $r_0<\zeta$ and (\ref{eq1 : definition of sigma and condition on F}) holds.  If $\sigma$ is an absolutely continuous curve, $V$ is smooth  such that $\sigma, V$ solve (\ref{eq1: conservation of mass dual space}) then $v,w,\zeta_t$  solve (\ref{eq1: conservation of mass physical}).
\end{lemma}\\

In the sequel, we define  ${\bf s}:[0, \infty)\times [0, H]\longrightarrow \mathcal{D}$ by 
 $${\bf s}(r,z)=\left( \frac{1}{2}\big( r_0^{-2}- r^{-2} \big), z \right). $$
 Note that ${\bf s} $ is invertible with inverse 
$${\bf d}(s,z)=  ({\bf d}_1 (s,z),{\bf d}_2(s,z))= \left(\frac{r_0}{\sqrt{1-2r_0^2s}} , z\right). $$

\begin{proposition}
Let $T>0$. Let $\varphi$ be smooth and define $P$ as in (\ref{eq:P in term of the pressure}). Assume $\varphi$ is such that  P is convex.  Let $ \Psi$ be the Legendre transform of $P$,  and $u, \theta, \zeta$ be smooth functions such that $r\chi_{\Gamma_{\zeta_t}}(r,z)$ is a probability measure for each $t$ fixed. Let ${\bf u}=(v,w)$ be a smooth velocity field  and set
\begin{equation} \label{eq1: unphisical  define V solution} 
  V_t\circ  \nabla P_t \circ {\bf s}:=\frac{D^{{\bf u}}}{D t} [\nabla P_t \circ {\bf s}].
\end{equation}
 If $ u$, $(v,w)$, $\varphi$, $\theta $ and  $\zeta$ satisfy (\ref{first1})-(\ref{fifth1}) and (\ref{eq: stability condition})  then $V$ is given by (\ref{eq1 unphysical1: velocity in dual space}). Furthermore, assume  $ \left\lbrace \sigma_t \right\rbrace $ is a family of a probability measures, absolutely continuous with respect to the Lebesgue measure such that $\nabla P_t\circ {\bf s}$ pushes $r\chi_{\Gamma_{\zeta_t}}(r,z)$ forward  to $\sigma_t$. If $( v,w)$, $\zeta$ solve (\ref{fourth1})and (\ref{eq: non physical free boundary}),  then 
  \begin{equation}\label{eq : cont eq} 
  \frac{\partial \sigma_t}{\partial t}+\div(V\sigma_t)=0 \qquad (0,T)\times \mathbb{R}^2
 \end{equation}
in the sense of distributions. \\
Conversely, let $\left\lbrace\sigma_t\right\rbrace_{t\in[0,T]}$ be a family of $\mathcal{P}(\mathbb{R}^2)$,  the set of Borel probability measures,  with supports in a fixed ball of  $\mathbb{R}^2_+$ and absolutely continuous with respect to the Lebesgue measure. Assume, for  each t fixed, that $(P_t^{\sigma_t},\Psi_t^{\sigma_t}, \rho_t^{\sigma_t})$ are smooth such that $(P_t^{\sigma_t},\Psi_t^{\sigma_t})$ are convex and Legendre transforms of each other and $\rho^{\sigma}>0$. Additionally, assume that  $\nabla P_t^{\sigma_t}$ and $\nabla \Psi_t^{\sigma_t}$ are inverse of each other  (on the interior of their domains)  and $(P_t^{\sigma_t},\Psi_t^{\sigma_t},\rho_t^{\sigma_t}, \sigma_t)$ satisfy (\ref{eq intro: Monge Ampere with parameter}) with

  \begin{equation}\label{eq : condition on free boundary on dual space} 
    P_t^{\sigma_t}(\rho_t^{\sigma_t}(H),H)=\Omega^2 r_0^2/ [2(1 -2r_0^2 \rho_t^{\sigma_t}(H))].
  \end{equation} 
 
  Define $\varphi^{\sigma}$, $\zeta^{\sigma}$, $u^{\sigma}$ and $ \theta^{\sigma}$ respectively through (\ref{eq:P in term of the pressure}), (\ref{eq: expressions of measure in physical space}) and
  
  \begin{equation} \label{eq1: unphisical  define u theta varphi solution} 
 ( u^{\sigma} r+r^2\Omega)^2=\partial_{s}P^{\sigma}\circ {\bf s},\qquad\frac{g}{\theta_0}\theta^{\sigma}=\partial_{z}P^{\sigma}\circ {\bf s}.
\end{equation}

 Let $V^{\sigma}$ be the velocity field as in (\ref{eq1 unphysical1: velocity in dual space}), when $\Psi$ is replaced by $\Psi^{\sigma}$, and let ${\bf u}^{\sigma}=(v^{\sigma},w^{\sigma})$ be the  velocity field  such that 
\begin{equation} \label{eq1: define u ^sigma} 
{\bf u}^{\sigma}\circ\nabla \Psi_t^{\sigma} \circ {\bf d}=\frac{D^{V^{\sigma}}}{D t} [\nabla \Psi_t^{\sigma} \circ {\bf d}].
\end{equation}
  
   If $\left\lbrace\sigma_t\right\rbrace_{t\in[0,T]}$ and $V^{\sigma}$ solve (\ref{eq00: Continuity equation unphysical problem}) then $(u^{\sigma}, v^{\sigma},w^{\sigma})$ $\theta^{\sigma}$ $\varphi^{\sigma}$ and $\zeta^{\sigma}$ solve (\ref{eq:non physical})-(\ref{eq: bdary cond on pressure}) and (\ref{eq: stability condition}).

\end{proposition}

\proof{}
We first note that if 

\begin{equation} \label{eq1: axisymmetric state} 
\nabla P_t\circ {\bf s}=\left[ ( u r+\Omega r^2)^2,g\frac{\theta}{\theta_0} \right] 
\end{equation} 
then
\begin{equation} \label{eq1: unphisical Material Derivative of the gradient of P} 
 \begin{aligned}
\frac{D}{Dt}[\nabla P_t\circ {\bf s}] &= \left[  2( u r+\Omega r^2)\frac{D}{Dt}( u r+\Omega r^2),\frac{g}{\theta_0}\frac{D \theta}{Dt} \right]\\
                                            &=\left[ 2(\sqrt{\partial_{s} P}\circ {\bf s} )(r\frac{{\rm D}  u}{{\rm D} t}+ u v +2r\Omega  v ),\frac{g}{\theta_0}\frac{D\theta}{Dt} \right].\\                                         
\end{aligned}
\end{equation}

Assume $u$, $\varphi$, $\theta $ satisfy (\ref{third1}) and (\ref{fifth1}). Then, in view of  (\ref{eq:P in term of the pressure}), a straightforward computation shows these equations are equivalent to (\ref{eq1: axisymmetric state}). Assume, in addition, that $( u,v,w)$, $\varphi$ and $\theta $   satisfy (\ref{first1}) (\ref{second1}). Then,  using (\ref{eq1: unphisical Material Derivative of the gradient of P}), we have

\begin{equation} \label{eq2: unphisical Material Derivative of the gradient of P} 
 \begin{aligned}
\frac{D}{Dt}[\nabla P_t\circ {\bf s}] =\left[ 2(\sqrt{\partial_{s} P}\circ {\bf s} )F,\;\frac{g}{\theta_0} S \right].                                             
\end{aligned}
\end{equation}

Moreover, if (\ref{eq: stability condition}) holds, we use the fact that $\Psi_t$ is the Legendre transform of $P_t$  to obtain that $\nabla P_t$ is invertible with inverse $\nabla \Psi$ so that, by combining  (\ref{eq1: unphisical  define V solution}) and (\ref{eq2: unphisical Material Derivative of the gradient of P}), we get (\ref{eq1 unphysical1: velocity in dual space}). Note that  if $( v,w)$, $\zeta$ solve (\ref{fourth1}) and (\ref{eq: non physical free boundary})  then (\ref{eq : cont eq}) follows directly from lemma \ref{le: conservation of total mass}.\\
 
 Conversely, note again that, as $\varphi^{\sigma}$ is defined through (\ref{eq:P in term of the pressure}), the equations in  (\ref{eq1: unphisical  define u theta varphi solution}) imply that $\varphi^{\sigma}$, $u^{\sigma}$ and $\theta^{\sigma}$  solve  (\ref{third1}) and (\ref{fifth1}). Since $\nabla \Psi^{\sigma}$ is invertible with inverse $\nabla P^{\sigma}$, (\ref{eq1: define u ^sigma}) can be rewritten as
 
 \begin{equation} \label{eq1: unphisical  define V[sigma] solution} 
  V_t^{\sigma}\circ  \nabla P^{\sigma}_t \circ {\bf s}:=\frac{D^{{\bf u}^{\sigma}}}{D t} [\nabla P^{\sigma}_t \circ {\bf s}].
\end{equation}

This, in light of (\ref{eq1: unphisical Material Derivative of the gradient of P}) and (\ref{eq1 unphysical1: velocity in dual space}) yields that $u^{\sigma}$, $v^{\sigma}$, $w^{\sigma}$ and $\theta^{\sigma}$ solve (\ref{first1}) (\ref{second1}). Note that $\rho^{\sigma}>0$ is equivalent to $\zeta^{\sigma}>r_0$.
As $P^{\sigma}$ is smooth and $\nabla P^{\sigma}$, $\nabla \Psi^{\sigma}$ are inverse of each other,(\ref{eq intro: Monge Ampere with parameter}) means that $\nabla P^{\sigma}_t$ pushes $e(s)\chi_{D_{\rho_t}}$ forward to $\sigma_t$, that is,  $\nabla P^{\sigma}_t\circ {\bf s}$ pushes $r\chi_{\Gamma_{\zeta^{\sigma}_t}}$  forward to $\sigma_t$. This, combined with (\ref{eq1: unphisical  define V[sigma] solution}), yields that  the equation (\ref{eq : cont eq}) with $V_t:= V_t^{\sigma}$, implies (\ref{fourth1}) and (\ref{eq: non physical free boundary}) by applying  lemma \ref{le: conservation of total mass}.  In view of (\ref{eq:P in term of the pressure}), we use (\ref{eq : condition on free boundary on dual space}) to obtain (\ref{eq: bdary cond on pressure}) \endproof



\section{Duality Methods and Monge-Ampere Problem}
\label{sec:Duality Methods and Monge-Ampere Problem}

In this section, we show the existence and uniqueness for the minimization problem \eqref{eq intro:primal1} by coming up with a dual problem. This provides a unique solution to the Monge Ampere equation (\ref{eq intro: Monge Ampere with parameter}) in some sense. Furthermore, this dual formulation helps  establish some regularity result for the domain $D_{\rho}$ in (\ref{eq intro: Monge Ampere with parameter}). 

Let $\sigma\in \calP_{[R_0]}$; we consider a system of PDEs, where the unknowns are
\begin{equation}\label{eq: dom P, Psi, h}
\Psi:[0,\infty)\times[0, \infty)\longrightarrow\mathbb{R},\qquad P:[0,1/2r_0^2)\times [0,H]\longrightarrow\mathbb{R},\qquad \rho:[0,H]\longrightarrow [0,1/2r_0^2).
\end{equation}
We impose that $\Psi$ and $P$ be  Legendre transforms of each other and that these functions solve the system of equations
\begin{equation}\label{eq:Monge Ampere}
 \begin{cases}
e(\frac{\partial \Psi}{\partial \Upsilon})\det(\nabla_{\Upsilon,Z}^{2}\Psi)= \sigma, \\
\nabla\Psi( spt(\sigma ))= D_{\rho},\\
P(\rho (z),z)={\Omega^2 r_0^2\over 2(1 -2r_0^2 \rho( z)) } \quad \hbox{on} \quad \{\rho >0\}.
\end{cases}
\end{equation} 

\begin{definition}\label{def: solution}
 (i) Assume that $\sigma=\rho\calL^2$.  Let  $P,\;\Psi,\;\rho $ be as in (\ref{eq: dom P, Psi, h}) such that $P,\;\Psi$ are  Legendre transforms of each other. We say that  $P$, $\Psi$ and $\rho$ solve equation (\ref{eq:Monge Ampere}) in a weak sense if 
\textbf{\textbf{}}\begin{equation}\label{eq: definition Monge Ampere solution}
 \begin{cases}
 \nabla\Psi_{\#}\sigma=e(s)\chi_{D_{\rho}}(s,z)\calL^2, \\
P(\rho(z),z)={\Omega^2 r_0^2\over 2(1 -2r_0^2 \rho( z)) } \quad \hbox{on} \quad \{\rho>0\}.
\end{cases}
\end{equation} 
(ii) We say that  $P$, $\Psi$ and $\rho$ solve equation (\ref{eq:Monge Ampere}) in the dual weak sense if 
\begin{equation}\label{eq: definition Monge Ampere solution1}
 \begin{cases}
 \nabla P_{\#}\left( e(s)\chi_{D_{\rho}}(s,z)\calL^2\right) =\sigma, \\
P(\rho(z),z)={\Omega^2 r_0^2\over 2(1 -2r_0^2 \rho( z)) } \quad \hbox{on} \quad \{\rho>0\}.
\end{cases}
\end{equation}

\end{definition}

Our main result  in this section is the following :\\

\begin{theorem}\label{thm : existence and uniqueness MA} 
 Let $\sigma\in \calP_{[R_0]}$ such that $\spt(\sigma)=\Delta$. Then, (\ref{eq:Monge Ampere})  admits a unique variational solution  $(\bar\Psi, \bar P, \bar \rho)$ in the sense that $(\bar\Psi, \bar P)$ is obtained as the unique maximizer in (\ref{eq:statement of dual problem}) and $\bar \rho$ is monotone, bounded away from $\frac{1}{2r_0^2}$ and obtained as the unique minimiser in (\ref{eq:primal1}).
 Moreover, if the support of $\sigma$ is contained in $[\frac{1}{R_0},R_0]\times [0,R_0]$ then $\partial D_{\bar \rho}$  is a finite union of the graphs of Lipschitz continuous functions.

\end{theorem}

\subsection{Primal and  Dual formulation of the problem}

Let $\sigma\in \mathcal{P}_{[R_0]}$. As discussed in the introduction, the Hamiltonian involves a certain functional that is  given by 
\begin{equation}\label{eq:forI} 
I(\rho,\gamma):=\int_{D_\rho \times \Delta} \left( -\langle p, q \rangle +{\Upsilon \over 2 r_0^2}-\Omega \sqrt \Upsilon + {r_0^2 \Omega^2 \over 2(1-2 r_0^2 s)}\right)  \gamma(dp, dq), 
\end{equation}
with $p=(s,z)$, $q=(\Upsilon,Z)$ and $\gamma \in \Gamma(e(s)\chi_{D_{\rho}}(s,z)\calL^2, \sigma)$.

 We recall that
 
$$ D_{\rho}=\left\lbrace (s,z): 0\leq s\leq \rho(z), z\in [0,H] \right\rbrace \qquad \text{ and } \qquad e(s)=r_0^4/(1-2s r_0^2)^2 \quad \text{for}\quad 0\leq 2r_0^2s <1.$$

 We consider the following minimization problem :

\begin{equation}\label{eq:primal1} 
\inf_{(\rho, \gamma)\in\frak{L}_{\sigma}} I(\rho,\gamma), 
\end{equation}
where 
 \begin{equation}\label{eq:forGamma}
  \frak{L}_{\sigma}=\Bigl\{(\rho,\gamma): \rho\in\calH_0  , \int_{D_{\rho}}e(s)dsdz=1, \gamma \in \Gamma(e(s)\chi_{D_{\rho}}\calL^2, \sigma)\Bigr\}.
 \end{equation}

$\calH_0$ is the set of all Borel measurable functions $\rho: [0,H]\longrightarrow [0,\frac{1}{2r_0^2}).$
To study the minimization problem in (\ref{eq:primal1}), we will introduce what will turn out to be its dual formulation by setting:
\begin{equation}\label{eq:forJ} 
J[\sigma](\Psi, P)= \int_{\mathbb{R}^2} \Bigl({\Upsilon \over 2r_0^2} -\Omega \sqrt \Upsilon- \Psi\Bigr) \sigma(dq) +j(P); \quad 
j(P)=\inf_{\rho \in \calH_0} \int_0^H \Pi_P(\rho(z),z)dz.
\end{equation}
 $J[\sigma]$ is defined on 
\begin{equation}\label{eq:for calU}
 \calU:=\Bigl\{ (\Psi, P)\in C(\Reals^2_+)\times C( \bar{\mathcal{D}}): P(p)+\Psi(q) \geq  \langle p, q \rangle \text{ for all } (p,q)\in\mathcal{D}\times\mathbb{R}_+^2 \Bigr\}.
\end{equation}

To $P: \mathcal{D}\longrightarrow \mathbb{R}$ we have associated  
\begin{equation}\label{eq:Pi1} 
\Pi_{P}(\rho,z)=\int_0^\rho \Bigl( \dfrac{\Omega^2 r_{0}^{2}}{2(1-2s r_{0}^{2})}-P(s,z) \Bigr) e(s)ds  \quad \hbox{for} \quad 0\leq 2r_0^2 \rho <1.
\end{equation}
  We observe that if $P_1\leq P_2$ then $\Pi_{P_1}\geq \Pi_{P_2}$   and also that if $P$ is a constant function that is equal to $C$ in (\ref{eq:Pi1}) then 
\begin{equation}\label{eq:Pi2} 
\Pi_{C}(\rho,z) :=\Pi_{P}(\rho,z)=\dfrac{\Omega^2 r_{0}^{6}(1-\rho r^{2}_{0})\rho}{2(1-2\rho r_{0}^{2})^{2}}-\dfrac{ C r_{0}^{4}\rho}{1-2\rho r_{0}^{2}}.
\end{equation} 
 The dual problem we will be looking at is the following:
\begin{equation}\label{eq:statement of dual problem} 
 \sup_{(\Psi,P)\in\calU} J[\sigma](\Psi,P).
\end{equation}


\subsection{ Existence of a minimizer for $\Pi_P(\cdot,z)$ and Twist condition}\label{subsection: Hypotheses}
To any $(P,\Psi)\in\calU$, we associate $(P^\ast,\Psi^\ast)$ defined by:

$$P^\ast(q)= \sup_{p \in\mathcal{D}} \left( \langle p, q \rangle-P(p)\right) \quad q \in \Reals_+^2 \qquad\hbox{ and }\qquad \Psi^\ast(p)= \sup_{q \in \Delta}\left(  \langle p, q \rangle-\Psi(q)\right) \quad p \in \mathcal{D}.$$

Let us denote by $\calU_0$ the subset of $\calU$  given by 
\begin{equation}\label{eq:LegendrepsiP} 
\calU_0 = \left\lbrace (P,\Psi)\in \calU: P(p)= \Psi^\ast(p)\quad p \in \mathcal{D}\quad \hbox{and}\quad \Psi(q)=P^\ast(q) \quad q \in \Reals_+^2\right\rbrace.
\end{equation} 
We note that if $(P,\Psi)\in\calU_0$ then $P$ and $\Psi$ are convex   as supremum of convex functions and 

\begin{equation}\label{eq:subgradients} 
\partial_\cdot P(p)\subset \Delta \text{ for any } p\in\mathcal{D}\qquad\hbox{ and }\qquad \partial_\cdot \Psi(q)\subset \mathcal{D} \text{ for any } q\in  \Reals_+^2.
\end{equation}


We consider  functions $P:\mathcal{D}\rightarrow \Reals$  such that
\begin{equation}\label{eq:conditionP} 
0\leq\frac{\partial P}{\partial z}(s,z)\leq R_0 \qquad\hbox{ and }\qquad 0 \leq\frac{\partial P}{\partial s}(s,z)\leq R_0.
\end{equation} 

\begin{lemma}\label{le:boundonheight1} Let $A\in \mathbb{R}_{+}$. Suppose $\hat{P}:\mathcal{D} \rightarrow \Reals$  such that $\hat{P}\leq \hat{P}(0,0)+A$. Then there exists a constant $M_{\hat{P}}$ depending on $\hat{P}(0,0)$  such that $2r_0^2M_{\hat{P}}<1$ and

\begin{equation}\label{eq:boundonheight1}
\sup_{0\leq z\leq H}\sup_{0 \leq 2r_0^2 \rho<1 } \{ \rho \; : \; \Pi_{\hat{P}}(\rho,z) \leq 0\} \leq M_{\hat{P}}.
\end{equation}
Furthermore, $M_{\hat{P}}$ is monotone nondecreasing in $\hat{P}(0,0)$, that is, if $\hat{P}_1, \hat{P}_2 $ satisfy the hypotheses above and are such that $\hat{P}_1(0,0)\leq \hat{P}_2(0,0) $ then  $M_{\hat{P}_1}\leq M_{\hat{P}_2}$.\\

\end{lemma} 
\begin{proof}{}
 We use $\hat{P}\leq \hat{P}(0,0)+A$ to establish that for any $z\in[0,H]$ fixed, 
$$
0\geq\Pi_{\hat{P}}(\rho,z)\geq \dfrac{\Omega^2 r_{0}^{6}(1-\rho r^{2}_{0})\rho}{2(1-2\rho r_{0}^{2})^{2}}-\dfrac{ (\hat{P}(0,0)+A) r_{0}^{4}\rho}{1-2\rho r_{0}^{2}}
$$
It is straightforward to check that
$$
M_{\hat{P}}:=\sup_{ 0\leq 2r_0^2 \rho<1} \left\lbrace  \rho \; : \;\dfrac{\Omega^2 r_{0}^{6}(1-\rho r^{2}_{0})\rho}{2(1-2\rho r_{0}^{2})}\leq  \Bigl[\hat{P}(0,0)+A\Bigr]r_{0}^{4}\rho
\right\rbrace 
$$ 
satisfies the above requirements.
\end{proof}

\begin{lemma}\label{le:boundonheight1*1}  Assume $ P_n,P: \mathcal{D}\rightarrow \Reals$   satisfy  the hypotheses in  Lemma \ref{le:boundonheight1} and are continuous. \\
\begin{enumerate}
\item[(i)] Given $z \in [0,H]$,   the set $\Arg_P(\cdot, z)$ consisting of $\rho$ minimizing $\Pi_P(\cdot, z)$ over  $[0,1/(2r_0^2))$ is non empty.
Moreover, 
\begin{equation}\label{eq:boundonheight12}
\bigcup_{0\leq z\leq H}  \Arg_P(\cdot, z)\subset [0,M_P],
\end{equation}

 where  $M_P$ is as in lemma \ref{le:boundonheight1}. \\

\item[(ii)] Suppose $\{P_n\}_{n=1}^\infty$  converges uniformly to $P$ on $\bar{\mathcal{D}}$. Then

\begin{equation}\label{eq:boundonheight2}
 2r^2_{0}\sup_{n}M_{P_n}<1.
\end{equation}
If, in addition,  $\{z_n\}_{n=1}^\infty \subset [0,H]$ converges to $z$ and we assume that  $\rho_n \in \Arg_{P_n}(\cdot, z_n)$ and that $\{\rho_n\}_{n=1}^\infty$ converges to $\rho$ then 

\begin{equation}\label{eq:boundonheight21}
\lim_{n \rightarrow \infty} \Pi_{P_n}(\rho_n, z_n) = \Pi_P(\rho, z) \qquad \text{ and }\qquad \rho \in \Arg_P(\cdot, z).
\end{equation}
In particular, for each $z \in [0,H]$ the set $\Arg_P(\cdot, z)$ is a compact subset of $\Reals.$ \\
\item[(iii)] Assume, in addition, that $P(\rho,\cdot)$ is Lipschitz and the first equation in (\ref{eq:conditionP}) holds a.e on $(0,H)$ for each $\rho $ fixed. Let $z_1, z_2 \in [0,H]$ be such that $z_1 \leq z_2$.    If $\rho_i \in \Arg_P(\cdot, z_i)$ $i=1,2$ then $\rho_1 \leq \rho_2.$ \\
\end{enumerate}
\end{lemma} 
\proof{} (i) Let $z\in [0,H]$. As $\Pi_P(0,z)=0$, in light of Lemma \ref{le:boundonheight1}, minimizing $\Pi_P(\cdot,z)$ over $[0,1/(2r_0^2))$ is equivalent to minimizing $\Pi_P(\cdot,z)$ over $[0,M_P].$ We observe that $\Pi_P(\cdot,z)$ is continuous on $[0,M_P].$ Hence, it admits a minimum there and $\Arg_P(\cdot, z) \subset [0,M_P].$ This establishes  (\ref{eq:boundonheight12}).\\
(ii) The convergence property of $\left\lbrace P_n\right\rbrace_{n=1}^{\infty}   $ ensures that $\{P_n (0,0)\}_{n=1}^\infty$ is bounded above by one of its terms say  $P_{n_0}(0,0)$ or $P(0,0)$. The monotonicity result in Lemma \ref{le:boundonheight1} ensures that $ M_{P_n}\leq M_{P_{n_0}}<\frac{1}{2r^{2}_{0}}$  or $ M_{P_n}\leq M_P<\frac{1}{2r^{2}_{0}}$  for all $n\geq 1$. Thus, (\ref{eq:boundonheight2}) holds. \\
Let $\{z_n\}_{n=1}^\infty \subset [0,H]$ be a sequence converging to $z$ and assume $\rho_n \in \Arg_{P_n}(\cdot, z_n)$ and is such that $\{\rho_n\}_{n=1}^\infty$ converges to $\rho$ and let $s \in [0, 1/(2r_0^2))$. We choose $M$  such that $M_{P},\sup_{n}M_{P_n},s \leq M<\frac{1}{2r^2_{0}}$ so that $K:=[0,M] \times [0,H]$ is a compact subset of $\mathcal{D}$. We observe that  $\{\Pi_{P_n}\}_{n=1}^\infty$ converges uniformly to $\Pi_P$ on $K$. This, combined with the fact that  $\rho_n$  minimizes $\Pi_{P_n}(\cdot, z_n)$, and   $\Pi_P$ is continuous, yields
 $$
\Pi_P(\rho, z)= \lim_{n \rightarrow \infty} \Pi_{P_n}(\rho_n, z_n) \leq \lim_{n \rightarrow \infty} \Pi_{P_n}(\varrho, z_n)=\Pi_P(\varrho, z). 
$$ 
Since this holds for any $s\in [0, 1/(2r_0^2))$, we have that $\rho \in \Arg_P(\cdot, z)$.\\
(iii) For each $z \in [0,H],$ $\Pi_P(\cdot, z)$ is differentiable on $(0,1/(2r_0^2))$ and its derivative is the integrand of $\Pi_P$. As $ P(s,\cdot)$ is Lipschitz , $\partial \Pi_P /  \partial s(s,\cdot)$ is differentiable almost everywhere on $(0,H)$ and 
\begin{equation}\label{eq:boundonheight5} 
{\partial^2 \Pi_P \over \partial z \partial s}(s,z)= -e(s) {\partial P \over \partial z}(s,z)\leq 0.
\end{equation}  
 We have used the first equation in (\ref{eq:conditionP}). This means that $\Pi_P$ satisfies a twist condition (see \cite{McCann}).
Let $z_i\in [0,H]$ and $\rho_i \in \Arg_P(\cdot, z_i)$ $i=1,2.$  We use the minimality  condition on $\rho_1$, $\rho_2$ and the fact that $P(s,\cdot)$ is Lipschitz to  obtain
\begin{equation}\label{eq:boundonheight6} 
0 \leq \Bigl(\Pi_P(\rho_2, z_1)-\Pi_P(\rho_1,z_1) \Bigr) +\Bigl(\Pi_P(\rho_1,z_2)-\Pi_P(\rho_2,z_2)\Bigr) =-\int_{\rho_1}^{\rho_2} ds \int_{z_1}^{z_2} {\partial^2 \Pi_P \over \partial z \partial s}(s,z) dz.
\end{equation} 
If $z_1 <z_2,$ then we use  (\ref{eq:boundonheight5}) and (\ref{eq:boundonheight6}) to get  $\rho_1 \leq \rho_2.$ \endproof

\begin{rmk}\label{re:boundonheight2} 
 Let $z \in [0,H]$ and $h \in \Arg_P(\cdot, z).$ If $\rho>0$ then $\partial \Pi /\partial \varrho(\rho,z) = 0$ that is,  
\begin{equation}\label{eq:relationPon} 
P(\rho,z) = {\Omega^2 r_0^2 \over 2(1-2\rho r_0^2)}.\\ 
\end{equation}
\end{rmk}

Let  $P: \mathcal{D}\rightarrow \Reals$ be such that $Argmin\Pi_{P}(\cdot,z)$ is compact for each $z\in[0,H]$ . We define 
$$ 
\rho^+(z)= \max_{\rho\in\Arg_P(\cdot, z)} \rho, \qquad \rho^-(z)= \min_{\rho\in\Arg_P(\cdot, z)} \rho
$$

\begin{lemma}\label{le:boundonheight3}
Assume $P$ satisfies the hypotheses in lemma \ref{le:boundonheight1*1} and the first equation in (\ref{eq:conditionP}). Then, the following hold:
\begin{enumerate}
\item[(i)] $\rho^-$ is lower semi-continuous and $\rho^+$ is upper semi-continuous. 
\item[(ii)] $\rho^-, \rho^+$ are monotone nondecreasing. 
\item[(iii)] Let $z_1, z_2 \in [0,H]$ be such that $z_1 \leq z_2$.   Then  $\rho^+(z_1) \leq \rho^-(z_2).$  
\item[(iv)]  $\rho^-$ is left continuous and $\rho^+$ is right continuous.
\item[(v)] Let $z\in [0,H)$. If $\rho^-$ is continuous at $z$ then $\rho^-(z) = \rho^+(z)$.\\
\end{enumerate}
\end{lemma}  

\proof{} (i) is a consequence of the continuity property in Lemma \ref{le:boundonheight1*1} (ii).
(ii) and (iii) come from Lemma \ref{le:boundonheight1*1} (iii). We use the fact that $\rho^-$ is monotone  nondecreasing and lower semi-continuous to obtain that  $ \rho^-$ is left continuous. A similar argument gives that $\rho^+$ is right continuous. Let $z_0\in [0,H)$ such that $\rho^-(z_0) <\rho^+(z_0)$. We note that, as $\rho^-$ is monotone nondecreasing, it has a right limit.  For $\delta>0$ small enough, we use Part (iii) to obtain that $\rho^+(z_0)\leq \rho^-(z_0+\delta)$ and so
$$\rho^-(z_0) <\rho^+(z_0)\leq \lim_{\delta\rightarrow 0^+}\rho^-(z_0+\delta).$$
This implies that $\rho^-$ is discontinuous at $z_0$ which proved (v).
\endproof

\begin{corollary}\label{co:uniqueheight1} There exists a countable set $\calN \subset [0,H]$ such that for every $z \not \in \calN$, $\Arg_P(\cdot, z)$ has a unique element.\\ 
\end{corollary} 

\begin{rmk}\label{re:huniqueinj} 
If  $P$ is Lipschitz and satisfies (\ref{eq:conditionP}) then $P$ satisfies the hypotheses in Lemma \ref{le:boundonheight1}. The compactness result in Lemma \ref{le:boundonheight1*1}(ii) combined with the definition of $\rho^-$ ensure that $\rho^-$ is  a minimizer  in the second equation of (\ref{eq:forJ}). Note that by (\ref{eq:boundonheight12}), 
\begin{equation}\label{eq:huniqueinj1} 
0\leq \rho^-(z)\leq M_P<\frac{1}{2r^2_0}
\end{equation}
for $z\in[0,H]$. Let $\left\lbrace P_n\right\rbrace^{\infty}_{n=1}\subset C(\mathcal{D}) $ be a sequence of Lipschitz functions uniformly convergent on $\mathcal{D}$ and satisfying (\ref{eq:conditionP}). By (\ref{eq:huniqueinj1}) and (\ref{eq:boundonheight2}), 
\begin{equation}\label{eq:huniqueinj2} 
0\leq \rho^-_n(z)\leq \sup_n M_{P_n}<\frac{1}{2r^2_0}
\end{equation}
for  $z\in[0,H]$ and all $n\geq 1$.
\endproof
\end{rmk} 

\subsection{Existence of a minimizer for the functional $I$}\hfill

\begin{rmk}\label{re:LipschitzPPsi} 
Let $( P,\Psi)\in \calU_0$. We recall that $P$ and $\Psi$ are Lipschitz and $P$ satisfies (\ref{eq:conditionP}). If in addition $P(0,0)=0$ then, in view of (\ref{eq:subgradients})
\begin{equation}\label{eq:LipschitzPPsi1} 
 |P(p)|\leq R_0(\frac{1}{2r_0^2}+H)=:R_0H_0.
\end{equation}
We note that $ 0\leq \langle p, q \rangle \leq  R_0H_0$ for $q\in\Delta$ and $p\in\mathcal{D}$. This, combined with the second equation in (\ref{eq:LegendrepsiP}) and (\ref{eq:LipschitzPPsi1}) yields that $\Psi $ is bounded on $\Delta$.  More precisely
$$-2 R_0H_0<-R_0H_0\leq \Psi(q) \leq 2H_0 R_0$$ 
$\text{ for } q \in \Delta$. As a consequence, 
$$
\Bigl| \int_\Delta \Bigl({\Upsilon \over 2r_0^2} -\Omega \sqrt \Upsilon- \Psi\Bigr)\sigma(dq) \Bigr| \leq 
 \int_\Delta \Bigl({\Upsilon \over 2r_0^2} +\Omega \sqrt \Upsilon\Bigr) \sigma(dq)+ 2 H_0R_0=: C(R_0)+ 2 H_0R_0 <\infty. 
$$
\end{rmk}

\begin{lemma} \label{le:aprioribound1} Let $C_0\in\mathbb{R}$. There exists $ C_1\in\mathbb{R} $ satisfying the following: whenever $(P, \Psi) \in \calU_0$ with $P(0,0)=0,$ $\lambda \in \Reals$ and $\sigma\in\calP_{[R_0]}$ are such that $-C_0\leq J[\sigma](P-\lambda,\Psi +\lambda) $ then $|\lambda| \leq C_1.$ \\
\end{lemma} 

\proof{} 
By (\ref{eq:LipschitzPPsi1}) $-R_0H_0<P(p)$ for $p\in \Delta_{r_0}$ so that  $$\Pi_{P-\lambda}\leq \Pi_{-H_0R_0-\lambda}. $$ 
Therefore, if $J[\sigma](\Psi +\lambda ,P-\lambda ) \geq -C_0$ then 
 
$$ 
-C_0 \leq -\lambda +  \int_\Delta \Bigl({\Upsilon \over 2r_0^2} -\Omega \sqrt \Upsilon- \Psi\Bigr) \sigma(dq) + \int_0^H \Pi_{-H_0R_0-\lambda}(\rho(z),z)dz
$$ 
for all $ \rho \in \calH_0.$ Hence, using $C(R_0)$ as  given  in Remark \ref{re:LipschitzPPsi}  and setting $\rho$ to be a constant function $\bar{\rho}_0$, we obtain 
$$ 
-C_0 \leq -\lambda + C(R_0)+ 2 H_0R_0+ \int_0^H \Pi_{-H_0R_0-\lambda}(\bar\rho_0,z)dz.
$$ 
 We use (\ref{eq:Pi2} )  to get 
$$-\bar C_0:=-C_0 -C(R_0)-2 H_0R_0\leq -\lambda + \dfrac{\Omega^2 r_{0}^{6}(1-\bar{\rho}_0 r^{2}_{0})\bar{\rho}_0H}{2(1-2\bar{\rho}_0r_{0}^{2})^{2}}-\dfrac{ (-H_0R_0-\lambda)H r_{0}^{4}\bar{\rho}_0}{1-2\bar{\rho}_0 r_{0}^{2}}.$$
 We rewrite this as

\begin{equation}\label{eq:bound4onP} 
\lambda \Bigl(1- H {r_0^4 \bar{\rho}_0 \over 1-2 \bar{\rho}_0 r_0^2} \Bigr) \leq \bar C_0+    H {r_0^4H_0R_0 \bar{\rho}_0 \over 1-2 \bar{\rho}_0 r_0^2} + H\dfrac{\Omega^2 r_{0}^{6}(1-\bar{\rho}_0 r^{2}_{0})\bar{\rho}_0}{2(1-2\bar{\rho}_0r_{0}^{2})^{2}}.
\end{equation}

Set $\bar{\rho}_0=0$ in (\ref{eq:bound4onP}). Then, 
\begin{equation}\label{eq:upperbound for lambda}
 \lambda\leq \bar C_0.
\end{equation}

When the constant value of $\bar \rho_0$ is chosen  in $[0,\frac{1}{2r_0^2})$ ( for instance close enough to $\frac{1}{2r_0^2}$ ) so that the factor of $\lambda$ in (\ref{eq:bound4onP} ) is negative then there exists a constant $\bar C_1$ such that 
\begin{equation}\label{eq:lowerbound for lambda}
 \lambda \geq \bar C_1.
\end{equation}

We combine (\ref{eq:upperbound for lambda}) and (\ref{eq:lowerbound for lambda}) to get the result.
 \endproof

\begin{lemma} \label{le:conv J(P)}

\begin{enumerate}
\item[(i)] Let $P\in  C(\bar{\mathcal{D}})$ be a Lipschitz function satisfying (\ref{eq:conditionP}). Then, $\rho^-$ is  the unique  minimizer in (\ref{eq:forJ}) ( up to a set of zero Lebesgue measure). \\

\item[(ii)]  Assume that $\left\lbrace P_n\right\rbrace^{\infty} _{n=1}\subset C(\bar{\mathcal{D}})$ is a sequence of Lipschitz functions satisfying (\ref{eq:conditionP}) such that $\left\lbrace P_n\right\rbrace^{\infty} _{n=1}$ converges uniformly to $P$.
Then 
\begin{equation}
j(P_n) \quad \text { converges to} \quad   j(P).
\end{equation}

\end{enumerate}

\end{lemma}

\proof{} The function $\rho^-$ is  a minimizer in (\ref{eq:forJ}) as stated in remark \ref{re:huniqueinj}.  Corollary \ref{co:uniqueheight1}  ensures the uniqueness  which proves (i). We note that as $\left\lbrace P_n\right\rbrace^{\infty} _{n=1} $ is uniformly Lipschitz and converges uniformly to $P$,  we have that $P$ is Lipschitz. Let $\rho^-_n$ be the minimizer in the second equation of (\ref{eq:forJ}) when $P$ is replaced by $P_n$. By Helly's theorem there exists a subsequence of $\left\lbrace \rho^-_n \right\rbrace^{\infty} _{n=1}$ that we denote again by $\left\lbrace \rho^-_n \right\rbrace^{\infty} _{n=1}$  and $\rho$ monotone nondecreasing such that $\left\lbrace \rho^-_n\right\rbrace^{\infty} _{n=1} $  converges to $\rho$ pointwise. It is straightforward  that $\left\lbrace \Pi_{P_n}\right\rbrace^{\infty} _{n=1}$ converges uniformly to $\Pi_P$ on compact subsets of $\Delta$. This, in view of  (\ref{eq:huniqueinj2}), (\ref{eq:boundonheight21}) and Part (i), yields $\rho^-=\rho\quad a.e.$ In light of  (\ref{eq:huniqueinj2}) again,  we next use the fact that $\left\lbrace P_n\right\rbrace^{\infty} _{n=1}$ is uniformly bounded to obtain  $\sup_{z\in[0,H]}\sup_n|\Pi_{P_n}(\rho^-_n(z),z)|<\infty$. A simple Lebesgue dominated convergence theorem yields (ii). \endproof

\subsection{A step towards the Proof of the main Theorem}\hfill

\begin{prop}\label{Proposition MA} 
Let $\sigma\in \calP_{[R_0]}$.
\begin{enumerate}
 \item[(i)] The set of maximizers $\calM$ of  $J[\sigma]$   over $\calU$ is such that $\calM\cap\calU_{0}$  is non empty ($\calU_{0}$ is defined by (\ref{eq:LegendrepsiP})).
 \item[(ii)] $I(\gamma,\rho)\geq J[\sigma](P,\Psi)$ for all $( P,\Psi)\in \calU_0$  and all  $(\gamma, \rho) \in\frak{L}_{\sigma}$. The equality holds if and only if  $\id \times \nabla P$ pushes $e(s)\chi_{D_{\rho}}(s,z)\calL^2$ forward to $\gamma$  
  and $\rho(z)$ minimizes $\Pi_P(.,z)$ for almost every $z\in [0,H]$. If, in addition, $\sigma$ is absolutely continuous with respect to the  Lebesgue measure, then the first condition for equality could be replaced by $\nabla\Psi\times\id $ pushes $\sigma$ forward to $\gamma$.
\item[(iii)] $I$ has a unique minimizer $(\gamma_0,\rho_0)$ over $\frak{L}_{\sigma}$. Moreover,  if $(P_0,\Psi_0)\in \calU_0$ maximizes $J[\sigma]$ on $\calU$ then $J[\sigma](P_0,\Psi_0)=I(\gamma_0,\rho_0) $ and $\rho_0$ is monotone non decreasing on $[0,H]$  satisfying (\ref{eq:huniqueinj1}) and 
\begin{equation} \label{eq 1: prop: P on the boundary}
2(1-2r^2_{0} \rho_0(z))P(\rho_0(z),z)=r^2_0\Omega^{2} \text{ on } \{\rho_0>0\}
\end{equation}
If, in addition, $\sigma$  is absolutely continuous with respect to the Lebesgue measure then $\nabla \Psi_0 \times \id$ pushes $\sigma$  forward to $\gamma_0$ and 
\begin{equation} \label{eq 1: inversibility of the gradient of P}
\nabla \Psi_0\circ\nabla P_0= \id \quad  e(s)\chi_{D_{\rho_0}}(s,z)\calL^2-\text{almost everywhere}  \quad \nabla P_0\circ\nabla \Psi_0= \id \quad \sigma-\text{almost everywhere}. 
\end{equation}
 \item[(iv)]   $J[\sigma]$ has a unique maximizer $(P_0,\Psi_0)$ on $\calU_0$ in the sense that if $J[\sigma](\Psi_0,P_0)= J[\sigma](\Psi_1,P_1)$ and $(\Psi_1,P_1)\in\calU_0$ then $P_1=P_0 \quad\text{on}\quad  D_{\rho_0} $ and $\Psi_1=\Psi_0\quad \text{on}\quad \Delta$.\\
\end{enumerate}
\end{prop}

\proof{}
1.  Set
$$
\bar{c}_0=\sup_{(p,q) \in \Delta \times \mathcal{D}} \langle p,q \rangle, \qquad \bar{P}_0(p)=\bar{c}_0, \quad \bar{\Psi}_0(q)=0
$$ 
so that $$(\bar{P}_0,\bar{\Psi}_0) \in \calU   \quad \text {   and }  \quad  -C_0:=J[\sigma](\bar{P}_0,\bar{\Psi}_0 )-1  \text{  is finite.} $$ 

Let $\{(\bar P_n,\bar \Psi_n )\}_{n=1}^\infty \subset \calU$ be a maximizing sequence for $J[\sigma]$ over 
$\calU$.\\
 We note that whenever $(\bar P,\bar \Psi)\in\calU$, by the double convexification trick (cfr. \cite{Villani} Page 51), we have 
$$(\bar P{\ast \ast},\bar P^\ast,)\in\calU_0 \hbox{ and } J[\sigma](\bar P,\bar \Psi) \leq J[\sigma]( \bar P{\ast \ast},\bar P^\ast).$$ 
This shows, on the one hand, that if the set of maximizers $\calM$ of $J[\sigma]$ over $\calU$ is non empty then  so is $\calM\cap\calU_0$  and, on the other hand, that we may assume without loss of generality that $\{(\bar P_n,\bar \Psi_n)\}_{n=1}^\infty$ is contained in $\calU_0.$ We set 
$$
\Psi_n=\bar \Psi_n+\bar P_n(0,0), \quad \lambda_n=-\bar P_n(0,0), \quad P_n=\bar P_n -\bar P_n(0,0). 
$$ 
We easily check that  $\{( P_n, \Psi_n)\}_{n=1}^\infty \subset \calU_0$ and 
$$
\lim_{n \rightarrow \infty} J[\sigma]( P_n-\lambda_n,\Psi_n+\lambda_n)=\lim_{n \rightarrow \infty} J[\sigma](\bar P_n,\bar\Psi_n)= \sup_{ \calU} J[\sigma].
$$
And so, for $n$ large enough  
\begin{equation}\label{eq:bound1onJ} 
-C_0 \leq J[\sigma](P_n-\lambda_n,\Psi_n+\lambda_n).
\end{equation} 
Therefore, as  $P_n(0,0)=0$ by Lemma \ref{le:aprioribound1} we obtain that $\{\lambda_n\}_{n=1}^\infty \subset \Reals$ is bounded. Hence, up to a subsequence, $\{\lambda_n\}_{n=1}^\infty$ converges to a real number $\lambda_\ast.$ \\
In view of (\ref{eq:subgradients}) and (\ref{eq:LipschitzPPsi1}), we have that the sequences $\left\lbrace P_n\right\rbrace^{\infty}_{n=1}\subset  C(\bar{\mathcal{D}})$ and $\left\lbrace \Psi_n\right\rbrace^{\infty}_{n=1}\subset  C(\bar\Delta)$  are uniformly bounded and uniformly Lipschitz. We then use  Ascoli- Arzela to conclude that there exists a subsequence of $\{(\Psi_n, P_n)\}_{n=1}^\infty$ converging  uniformly to some  $(\Psi_{\ast}, P_{\ast})\in C(\Delta)\times C(\bar{\mathcal{D}})$. In the sequel, we assume without loss of generality that 
$$\{\lambda_n\}_{n=1}^\infty \text{  converges to }\lambda_\ast \hbox{  and } \{( P_n,\Psi_n)\}_{n=1}^\infty \text{  converges uniformly to } (P_{\ast},\Psi_{\ast}). $$
We set 
$$ P_0:=P_{\ast}-\lambda_\ast, \quad  \Psi_0 :=\Psi_{\ast}+ \lambda_\ast.$$
Therefore  $$\{(\bar P_n,\bar \Psi_n)\}_{n=1}^\infty \text{  converges uniformly to }  (P_0,\Psi_0 ).$$
 Note that $\{\bar P_n\}_{n=1}^\infty$ are Lipschitz and satisfies (\ref{eq:conditionP}). We use the fact that $\{\bar \Psi_n)\}_{n=1}^\infty \text{  converges uniformly to }  \Psi_0$, $\sigma$ is a probability measure  and lemma \ref{le:conv J(P)} (ii) to obtain that
$$\{J[\sigma](\bar P_n,\bar \Psi_n)\}_{n=1}^\infty \text{  converges to }   J[\sigma](P_0,\Psi_0).$$
This established that $(\Psi_0, P_0) $ is a maximizer of $J[\sigma]$ over $\calU.$
 
2. Let $( P,\Psi)\in \calU_0$  and  $(\gamma, \rho) \in\mathcal{L}_{\sigma}$. Then $\Psi, P$  are Lipschitz  and  $P(p) +\Psi(q) \geq \langle p, q\rangle$. We note that 
\begin{equation}\label{eq:inequalityjPi2} 
\begin{split}
j(P) \leq \int_0^H \Pi_P(\rho(z),z) dz=\int_{D_\rho } \Bigl( {\Omega^2 r_0^2 \over 2(1-2s r_0^2)}-P(s,z)\Bigr) e(s)dsdz.\\
\end{split}
\end{equation}  

As $\gamma \in \Gamma \left( e(s)\chi_{D_\rho}\calL^2, \sigma\right) $ and $P(p) +\Psi(q) \geq \langle p, q\rangle$ we use (\ref{eq:inequalityjPi2}) to get 
\begin{equation}\label{eq:inequalityjPi4} 
 J[\sigma]( P,\Psi)\leq \int_{D_{\rho}\times \Reals^2_+ }\left( -\langle p, q \rangle +{\Upsilon \over 2 r_0^2}-\Omega \sqrt \Upsilon + {r_0^2 \Omega^2 \over 2(1-2 r_0^2 \rho)}\right) \gamma(dp, dq)= I(\rho,\gamma).
\end{equation}

 Note that equality holds in (\ref{eq:inequalityjPi4}) if and only if equality holds in (\ref{eq:inequalityjPi2}) and $P(p) + \Psi(q)= \langle p, q\rangle$ for $\gamma$ almost every $(p,q)$ .   The first condition means that $\rho(z) \in \Arg_P(\cdot, z)$ for almost every $z \in [0,H]$ by using Lemma \ref{le:conv J(P)} (i). As the first projection of $\gamma$ is absolutely continuous with respect to $\calL^2$, the second condition means that $q=\nabla P(p)$ for $\gamma$ almost every $(p,q)$ and so,  $\gamma$ is concentrated on the graph of $\nabla P$. This implies that $\id \times \nabla P$ pushes $e(s)\chi_{D_{\rho}}(s,z)\calL^2$ forward to $\gamma$. 

3.
Assume that   $(P_0, \Psi_0)\in\calU_{0}$ is a maximizer of $J[\sigma]$ over $\calU$. Let $ g \in C_c(\Reals^2)$. For  any $\delta \in (-1,1)$, we set
 $$\Psi_\delta=\Psi_0 + \delta g \qquad \hbox{ and }\qquad  P_{\delta}=\Psi_\delta^*.$$ 
 We note that $ \left\lbrace P_{\delta}\right\rbrace_{-1<\delta<1}\subset C(\bar{\mathcal{D}}). $ It can be shown that (cfr. \cite{Gangbo95a} \cite{Gangbo95})
\begin{equation}\label{eq:controlPe1} 
||P_\delta-P_0||_\infty \leq |\delta| ||g||_\infty \qquad \hbox{ and }\qquad  \lim_{\delta \rightarrow 0} {P_\delta(p)-P_0(p) \over \delta}= -g(\nabla P_0(p)) 
\end{equation} 
for all $p \in \dom (\nabla P_0)$. As $P_0$ is Lipschitz, the second equation in (\ref{eq:controlPe1}) holds almost everywhere with respect to $\calL^2.$\\
Fix $z\in[0,H].$ Let $\left\lbrace \delta_{n}\right\rbrace^{\infty} _{n=1}\subset(-1,1)$  converging to $0$.  We note that the first equation in (\ref{eq:controlPe1}) ensures that $\left\lbrace P_{\delta_n}\right\rbrace_{n=1}^{\infty}$ uniformly converges to $P_0.$  For each $n\geq 1,$
 as $\Arg_{P_{\delta_n}}(\cdot, z)$ is compact (cfr Lemma \ref{le:boundonheight1*1}), let $\rho_{\delta_{n}}(z)$ denote its  smallest element . For the same reasons, let $\rho_0(z)$ denote the smallest element of $\Arg_{P_{0}}(\cdot, z)$.  Then $\left\lbrace \rho_{\delta_{n}}(z)\right\rbrace_{1}^{\infty}$ is bounded in light of (\ref{eq:huniqueinj2}) and so, without loss of generality we assume that $\left\lbrace \rho_{\delta_{n}}(z)\right\rbrace_{n=1}^{\infty}$ converges. If $z$ is a continuity point for $\rho_{0}$ then  Lemma \ref{le:boundonheight3} (v) ensures that  $h_0(z)$ is the unique element of $\Arg_{P_0}(\cdot, z)$  and so by using  Lemma \ref{le:boundonheight1*1} (ii) we obtain
$$\lim_{n \rightarrow \infty} \rho_{\delta_n}(z)= \rho_0(z).$$
As $\left\lbrace \delta_{n}\right\rbrace^{\infty} _{n=1}$ is arbitrary, we obtain denoting by $\rho_{\delta}(z)$ the smallest element of $\Arg_{P_{\delta}}(\cdot, z)$
\begin{equation}\label{eq:limithe1} 
\lim_{\delta \rightarrow 0} \rho_\delta(z)= \rho_0(z).
\end{equation} 
In light of Corollary \ref{co:uniqueheight1}, the equation  (\ref{eq:limithe1} ) holds for almost every $z \in [0,H]$. \\
Fix $\delta \in (-1,1)$. By definition of $\rho_0(z)$ we have  $\Pi_{P_0}(\rho_0(z),z) \leq \Pi_{P_0}(\rho_\delta (z),z) $ and so
\begin{equation}\label{eq:forPi1} 
\Pi_{P_0}(\rho_0(z),z)-\Pi_{P_\delta}(\rho_\delta(z),z)\leq   \Pi_{P_0}(\rho_\delta(z),z)-\Pi_{P_\delta}(\rho_\delta(z),z)=    \int_0^{\rho_\delta(z)} (P_\delta(s,z)-P_0(s,z))e(s) ds.
\end{equation} 
 Similarly, we establish that 
\begin{equation}\label{eq:forPi2} 
\Pi_{P_\delta}(\rho_\delta(z),z)-\Pi_{P_0}(\rho_0(z),z)\leq   \Pi_{P_\delta}(\rho_0(z),z)-\Pi_{P_0}(\rho_0(z),z)=   - \int_0^{\rho_0(z)} (P_\delta(s,z)-P_0(s,z))e(s) ds.
\end{equation} 
Let again $\left\lbrace \delta_{n}\right\rbrace^{\infty} _{n=1}\subset(-1,1)$  converging to $0$. We use the definition of $j$ in (\ref{eq:forJ}),  with (\ref{eq:forPi1} ), (\ref{eq:forPi2}) to obtain that 
\begin{equation}\label{eq:forj1}
\int_0^H \int_0^{\rho_0(z)} (P_{\delta_n}(s,z)-P_0(s,z))e(s) dsdz\leq j(P_0)-j(P_{\delta_n})\leq \int_0^H \int_0^{\rho_{\delta_{n}}(z)} (P_{\delta_n}(s,z)-P_0(s,z))e(s)dsdz.
\end{equation}
 In view of (\ref{eq:controlPe1}), $\left\lbrace P_{\delta_n}\right\rbrace^{\infty}_{n=1}$ converges uniformly to $P_0$, so   (\ref{eq:boundonheight2})  holds for $\left\lbrace P_{\delta_n}\right\rbrace^{\infty}_{n=1}$. We then choose $M$ such that
$$2r_0^2\sup_{n} M_{P_{\delta_n}}\leq 2r_0^2 M<1.$$
We note that $\partial_{\cdot}P_{\delta_n}\subset \Delta$ and so $\left\lbrace P_{\delta_n}\right\rbrace^{\infty}_{n=1}$ is uniformly Lipschitz  and satisfies (\ref{eq:conditionP}) .  By (\ref{eq:huniqueinj2}), 
\begin{equation}\label{eq:hdeltabounded}
0 \leq \rho_{\delta_n}(z)\leq M
\end{equation}
for  $z\in[0,H]$ a.e and $n\geq 1$. This ensures that the integrals in (\ref{eq:forj1}) are finite  for $n\geq 1$.\\
We rewrite (\ref{eq:forj1}) as 
\begin{equation}\label{eq:forj2}
\begin{split}
 0\leq j(P_0)-j(P_{\delta_n})-\int_0^H dz\int_0^{\rho_0(z)} (P_{\delta_n}(s,z)-& P_0(s,z))e(s) ds \\
 &\leq \int_0^H \int_{\rho_0(z)}^{\rho_{\delta_n}(z)} (P_{\delta_n}(s,z)-P_0(s,z))e(s) ds dz.\\
\end{split}
\end{equation}
  We use  the fact $e$ is bounded on $[0,M]$, the first equation in (\ref{eq:controlPe1}) and apply the Lebesgue dominated convergence theorem, using  (\ref{eq:limithe1}) and (\ref{eq:hdeltabounded}) to obtain that 

\begin{equation}\label{eq:forj3}
 \limsup_{n\rightarrow\infty}\Big|\int_0^H dz\int_{\rho_0(z)}^{\rho_{\delta_n}(z)} \dfrac{(P_{\delta_n}(s,z)-P_0(s,z))e(s)}{\delta_n} ds \Big|\leq \max_{[0,H]}e||g||_{\infty}\limsup_{n\rightarrow\infty}\int_0^H |\rho_{\delta_n}(z)-\rho_0(z)|dz=0.
\end{equation}

By the Lebesgue dominated convergence theorem, (\ref{eq:controlPe1}) implies that 
\begin{equation}\label{eq:forj4}
\lim_{n\rightarrow \infty} \int_0^H dz\int_0^{\rho_0(z)} \dfrac{(P_{\delta_n}(s,z)-P_0(s,z))e(s)}{\delta_n} ds =-\int_0^H \int_0^{\rho_0(z)}g(\nabla P(p))e(s)dsdz.
\end{equation}

We note that 
\begin{equation}\label{eq: variation for J}
{J[\sigma](P_{\delta_n}, \Psi_{\delta_n})-J[\sigma](P_0, \Psi_0) \over \delta_n}=  -\int_\Delta g d\sigma + \dfrac{j(P_{\delta_n})-j(P_0)}{\delta_n}.
\end{equation}
 We use the fact that $\left\lbrace \delta_{n}\right\rbrace^{\infty} _{n=1}$ is an arbitrary sequence that converges to $0$ and combine (\ref{eq:forj2})-(\ref{eq: variation for J}) to get 

\begin{equation}\label{eq:limitJje11}   
\lim_{\delta \rightarrow 0}  {J[\sigma](P_\delta, \Psi_\delta)-J[\sigma](P_0, \Psi_0) \over \delta}= -\int_\Delta g d\sigma +  \int_0^H dz \int_0^{\rho_0(z)} g(\nabla P(s,z)) e(s) ds. 
\end{equation}  
Since $(P_0,\Psi_0)$ maximizes $J[\sigma]$ over $\calU$  and  $(P_\delta, \Psi_\delta) \in \calU $, (\ref{eq:limitJje11}) implies that  
\begin{equation}\label{eq:limitJje12}   
\int_\Delta g d\sigma= \int_0^H dz \int_{0}^{\rho_0(z)} g(\nabla P_0(p)) e(s) dp. 
\end{equation} 
 (\ref{eq:limitJje12}) holds for any $g \in C_c(\Reals^2)$ which means that $\nabla  P_0$ pushes $e(s)\chi_{D_{\rho_0}}(s,z)\calL^2$ forward to $\sigma$.
 
4. We recall that $(P_0, \Psi_0)\in\calU_{0}$  is a maximizer of $J[\sigma]$ over $\calU$,  that  for $z\in[0,H]$, $\rho_0(z)$ is the smallest element of $\Arg_{P_0}(\cdot, z)$ and we set 
                        $$\gamma_0:=(\id\times\nabla P_0)_{\#}e(s)\chi_{D_{\rho_0}}(s,z)\calL^2. $$

 Then, by part (ii) of the theorem, we have $I(\rho_0,\gamma_0)=J[\sigma](P_{0},\Psi_{0})$ which ensures that $(\rho_0,\gamma_0)$ is a minimizer in (\ref{eq:primal1}).\\
 Let $(\bar{\rho},\bar{\gamma})$ be another minimizer in (\ref{eq:primal1}). Then $I(\bar{\rho},\bar{\gamma})= I(\rho_0,\gamma_0)$ and so $I(\bar{\rho},\bar{\gamma})=J[\sigma](P_{0},\Psi_{0})$. Again by part (ii) $(\id\times\nabla P_0)$ pushes forward $e(s)\chi_{D_{\bar \rho}}\calL^2$ to  $\bar{\gamma}$  and $\bar{\rho}(z)\in\Arg_{P_0}(\cdot, z)$ for a.e $z\in [0,H]$. We use Corollary \ref{co:uniqueheight1} to obtain that $\bar{\rho}(z)=\rho_{0}(z)$ a.e. These prove that the minimizer in (\ref{eq:primal1}) is unique. By Remark \ref{re:boundonheight2}, equation (\ref{eq 1: prop: P on the boundary}) holds. In light of lemma \ref{le:conv J(P)}(i), $\rho_0$ is monotone and satisfies (\ref{eq:huniqueinj1}). The equation (\ref{eq 1: inversibility of the gradient of P}) is well known (see \cite{Villani}).
5. Assume $(P_1,\Psi_1)$ is another maximizer of $J[\sigma]$ in $\calU_0$. Then 
$$\gamma_0=(\id\times\nabla P_0)_{\#}e(s)\chi_{D_{\rho_0}}(s,z)\calL^2 =(\id\times\nabla P_1)_{\#}e(s)\chi_{D_{\rho_0}}(s,z)\calL^2 .$$ This implies that $\nabla P_0=\nabla P_1$ $e(s)\chi_{D_{\rho_0}}(s,z)\calL^2$-a.e and so the equality holds $\calL^2$-a.e on $D_{\rho_0}$ as $e>0$. As, $D_{\rho_0}$ is connected and $P_0$ and $P_1$ are Lipschitz continuous satisfying (\ref{eq 1: prop: P on the boundary}), we conclude that $P_1=P_0$ on $D_{\rho_0}$ and without loss of generality we take $P_1=P_0$ on $\mathcal{D}$. Consequently, $\Psi_1=\Psi_2$ on $\Delta$.
\endproof


\subsection{ Regularity property of the domain $D_{\rho}$}
\label{ sec: Regularity property of the domain}
In this section we consider  the Lipschitz functions $P$   that satisfy
\begin{equation}\label{eq:conditionP for regularity of the domain} 
0\leq\frac{\partial P}{\partial s}(s,z)\leq R_0 \qquad\hbox{ and }\qquad \frac{1}{R_0} \leq\frac{\partial P}{\partial z}(s,z)\leq R_0.
\end{equation} 
As a consequence,  $\Arg_P(\cdot, z)$ is compact. We recall that for such $P$,
\begin{equation}\label{eq: definition h+ h- for boundary} 
 \rho^+(z)= \max_{\rho\in\Arg_P(\cdot, z)} \rho, \qquad \rho^-(z)= \min_{\rho\in\Arg_P(\cdot, z)} \rho.
\end{equation}

\begin{lemma}\label{lemma existence of h}
  Assume $P$ is Lipschitz and satisfies (\ref{eq:conditionP for regularity of the domain}).
The set $\calZ=\left\lbrace z\in [0,H]: 0\in Argmin \Pi_{P}(\cdot,z)\right\rbrace $  when non empty, is a closed interval of the form $[0,z^*]$. In the case $\calZ$ is empty we set $z^*=0.$\\
\end{lemma}
\proof{}
Assume $\calZ$ is non empty and set $z^*$ to be its supremum. By definition of $z^*$, $\calZ\subset [0,z^*].$ Conversely, let $\left\lbrace z_n\right\rbrace _{n=1}^{\infty}$ be a sequence in $\calZ$ such that $\left\lbrace z_n\right\rbrace _{n=1}^{\infty}$ converges to $z^*$. Then, we use Lemma \ref{le:boundonheight1*1} (ii) to obtain that $0\in Argmin \Pi_{P}(\cdot,z^*)$, that is, $z^*\in\calZ.$ Lemma \ref{le:boundonheight1*1} (iii) ensures that $[0,z^*)\subset \calZ$.\endproof
\begin{lemma}\label{lem: existence of h inverse 1}
 Assume $P$ is Lipschitz and satisfies (\ref{eq:conditionP for regularity of the domain}) and let $z^*$ be as in Lemma \ref{lemma existence of h}.\\
\begin{enumerate}
 \item [(i)] There exists $c_0 >0$ such that if  $z^*\leq z_1\leq  z_2 \leq H$   and  $\rho_i\in Argmin\Pi_{P}(\cdot, z_i)$ $ i=1,2$, then  
\begin{equation}\label{eq:lip property of h}
 z_2-z_1\leq c_0(\rho_2-\rho_1).
\end{equation}
 \item [(ii)]  For any $z_1, z_2\in [z^*,H]$, $Argmin\Pi_{P}(\cdot, z_1)\cap Argmin\Pi_{P}(\cdot, z_2)=\emptyset$ if $z_1\neq z_2$ 
 \item [(iii)] $\rho^-$ , $\rho^+$ are strictly increasing on  $ [z^*,H]$.
\end{enumerate}
\end{lemma}
\proof{} Let $m(s)= {\Omega^2 r_0^2 \over 2(1-2s r_0^2)}.$ Note that $m$ is Lipschitz continuous on $[0,M_P]$. Here, $M_P$ is defined as in Lemma \ref{le:boundonheight1}.\\
Set 
$$
\alpha(s, z)= m(s)-P(s, z).
$$ 
As $P$ satisfies the first equation in (\ref{eq:conditionP for regularity of the domain} ), we have
 
\begin{equation}\label{eq1: Existence of h}
- Lip(m)- R_0 \leq \partial_{s}\alpha(s,z) \leq Lip(m).
\end{equation} 
Let $z_1, z_2\in (z^*,H]$  such that $z_1<z_2$ and $\rho_i\in Argmin\Pi_{P}(\cdot, z_i)$ $ i=1,2.$
Remark \ref{re:boundonheight2}  ensures that $\alpha(\rho_2,z_2)=\alpha(\rho_1,z_1)=0$ and so 
\begin{equation}\label{eq2: Existence of h}
\alpha(\rho_2,z_1)-\alpha(\rho_2,z_2)=  \alpha(\rho_2,z_1)-\alpha(\rho_1,z_1) .
\end{equation} 
 We exploit the second equation in  (\ref{eq:conditionP for regularity of the domain}) to obtain that
\begin{equation}\label{eq3: Existence of h}
\alpha(\rho_2,z_1)-\alpha(\rho_2,z_2)= \int_{z_2}^{z_1} \partial_z \alpha(\rho_2,z)dz=   \int_{z_1}^{z_2} \partial_z P(\rho_2,z)dz \geq \frac{1}{R_0} (z_2-z_1). 
\end{equation}  
 The second inequality in (\ref{eq1: Existence of h}) leads to
\begin{equation}\label{eq4: Existence of h}
\alpha(\rho_2,z_1)-\alpha(\rho_1,z_1)= \int_{\rho_1}^{\rho_2} \partial_s \alpha(s,z_1)ds \leq Lip(m)  (\rho_2-\rho_1). 
\end{equation} 
We combine (\ref{eq2: Existence of h}- \ref{eq4: Existence of h}) to conclude that 
$$
 (z_2-z_1)\leq R_0 Lip(m)( \rho_2-\rho_1)=c_0(\rho_2-\rho_1)
$$
 for all $z^*< z_1\leq  z_2 \leq H$. Note that if $\calZ=\emptyset$, the argument above is  still valid when $z_1=z^*.$ In the sequel, we assume that  $\calZ\neq\emptyset$. To obtain the inequality (\ref{eq:lip property of h}) when $z_1=z^*$, we consider a sequence $\left\lbrace \bar{z}^n\right\rbrace^{\infty}_{n=1}$ in $(z^*,H]$ such that $\bar{z}^n>z^*$ and  $\left\lbrace \bar{z}^n\right\rbrace^{\infty}_{n=1}$ converges to $z^*$. Let $\rho^n\in Argmin\Pi_{P}(\cdot, \bar{z}^n)\subset [0,M_P]$ so that (\ref{eq:lip property of h}) holds for $(z_1,\rho_1)$ and $(\bar{z}^n, \rho^n)$. Since $Argmin\Pi_{P}(\cdot, \bar{z}^n)\subset [0,M_P]$, we assume without loss of generality that sequence $\left\lbrace \rho^n\right\rbrace^{\infty}_{n=1}$ converges. As $\left\lbrace \bar{z}^n\right\rbrace _{n=1}^{\infty}$ converges to $z^*$, Lemma \ref{le:boundonheight1*1} (ii) ensures that $\left\lbrace \rho^n\right\rbrace^{\infty}_{n=1} $ converges to an element of $ Argmin\Pi_{P}(\cdot, z^*)$. We let $n$ go to $\infty$ in  $ z_2-\bar{z}^n\leq c_0(\rho_2-\rho^n)$ to obtain the desired result.
 (ii) and (iii) follow directly from (\ref{eq:lip property of h}). \endproof

Let $P$ be Lipschitz and satisfy (\ref{eq:conditionP for regularity of the domain}) and $\rho^-$ be as in (\ref{eq: definition h+ h- for boundary}). We define 

\begin{equation}\label{eq: rmk : boundary h: 00} 
\mathfrak{a}(s):=\inf A(s) \qquad \text{with} \qquad  A(s):=\left\lbrace z\in[z^{*},H]: \rho^{-}(z)\geq s\right\rbrace
\end{equation}

for $s\in[0, \rho^-(H)]$.

\hfill
\begin{rmk}\label{rmk: inverse function on intervals}
\begin{enumerate}
\item Let $0\leq s_1 \leq s_2\leq \rho^-(H)$. If $\rho^-(z)\geq s_2$ then $\rho^-(z)\geq s_1$ and so $A(s_2)\subset A(s_1)$.\\
\item Let $\varepsilon >0$ small enough. By Lemma \ref{le:boundonheight1*1} (iii), $\rho^-(z^*+\varepsilon)\geq \rho^+(z^*)$ and so $z^*+\varepsilon\in A( \rho^+(z^*))$. As $z^*$ is a lower bound for  $A( \rho^+(z^*))$, we conclude that $z^*\leq \inf A(\rho^+(z^*))\leq z^*+\varepsilon $. By the arbitrariness of $\varepsilon$ we obtain  $z^*=\mathfrak{a}(\rho^+(z^*))$.\\
\item Let $\rho^-(z^*)\leq s \leq \rho^+(z^*)$. By part (1) of this remark, $A(\rho^-(z^*))\subset A(s)\subset A(\rho^+(z^*))$ and so $\mathfrak{a}(\rho^-(z^*))\leq \mathfrak{a}(s)\leq \mathfrak{a}(\rho^+(z^*))$. We easily checked that $\mathfrak{a}(\rho^-(z^*))=z^*$. In view of part (2) of this remark, we obtain that $\mathfrak{a}(\rho^-(z^*))= \mathfrak{a}(s)= \mathfrak{a}(\rho^+(z^*))=z^*$.
\end{enumerate}
\end{rmk}
\hfill

\begin{lemma}\label{lem: inverse function on intervals}
 Assume $P$ satisfies (\ref{eq:conditionP for regularity of the domain}). Let $z^*$ be as in Lemma \ref{lemma existence of h} such that $z^*<H$. 
\begin{enumerate}
 \item [(i)] $\mathfrak{a}$ is non decreasing on $\in[0,h^{-}(H)]$. 
 \item [(ii)] If $s\in (\rho^{+}(z^{*}),\rho^{-}(H))$ then $\mathfrak{a}(s)$ is an interior point of $[z^*, H]$. 
\item [(iii)] If $s\in[\rho^-(z^*),\rho^{-}(H))$ then   $ s\in [\rho^-(\mathfrak{a}(s)),\rho^+(\mathfrak{a}(s))]$. Moreover, if $s\in[\rho^{-}(z),\rho^{+}(z)]$ for some $z\in[z^*,H)$ then $\mathfrak{a}(s)=z.$\\
\end{enumerate}
\end{lemma}

\proof{}
(i) is immediate from remark \ref{rmk: inverse function on intervals} (1).\\
 Let $s\in (\rho^{+}(z^{*}),\rho^{-}(H))$. As $s<\rho^{-}(H)$,
 we have that $H\in A(s)$. We next choose $ z\in(\mathfrak{a}(s), H]$. The characterization of the infimum in $(\ref{eq: rmk : boundary h: 00})$ ensures that there exists $\bar z\in A(s)$ such that $\mathfrak{a}(s)\leq \bar z<z$. $\bar z\in A(s)$ implies that  $\rho^{-}(\bar z)\geq s$ and  as $\rho^-$ increasing, $\rho^-(\bar z)\leq \rho^-(z)$. We conclude that $\rho^{-}(z)\geq s$ and so $z\in A(s)$. Hence $ (\mathfrak{a}(s),H]\subset A(s)$.
We next show that $\mathfrak{a}(s)$ is an interior point of the interval $[z^{*},H]$.

Let $\left\lbrace a_{n}\right\rbrace^{\infty}_{n=1}$ be a sequence in $(z^*,H)$ such that $\left\lbrace a_{n}\right\rbrace^{\infty}_{n=1}$  converges to $z^{*} $. We use the right continuity of $\rho^{+}$ ( cfr Lemma \ref{le:boundonheight3} (iv)) to obtain that $\left\lbrace \rho^{+}(a_{n})\right\rbrace^{\infty}_{n=1}$ converges to $ \rho^{+}(z^{*})$. As $s>\rho^{+}(z^{*})$  we obtain  
\begin{equation}\label{eq: rmk : boundary h: 01} 
s>\rho^{+}(a_{n})>\rho^{+}(z^{*})
\end{equation}
for $n$ large enough. We next choose $a_{n}$ in (\ref{eq: rmk : boundary h: 01}) to be  points of continuity of $\rho^{-}$  so that $\rho^{+}(a_{n})=\rho^{-}(a_{n})$ ( cfr Lemma \ref{le:boundonheight3} (v)). Therefore (\ref{eq: rmk : boundary h: 01}), becomes
\begin{equation}\label{eq: rmk : boundary h: 02} 
s>\rho^{-}(a_{n})>\rho^{+}(z^{*})
\end{equation}
for $n\geq n_0$  for some $n_0\in\mathbb{N}$. In light of the definition of $A(s)$,  the first inequality in (\ref{eq: rmk : boundary h: 02}) implies that $a_n\in (z^*,H)\setminus A(s)$  and so  in view of (\ref{eq: rmk : boundary h: 00}),   $a_{n}\leq\frak{a}(s)$ for all $n\geq n_0$. Since $\left\lbrace a_n\right\rbrace^{\infty}_{n=1} $  converges to $z^{*}$, there exists $p_0>n_0$ such that $a_{p_0}<\mathfrak{a}(s)$. The second inequality in (\ref{eq: rmk : boundary h: 02}) implies that   $\rho^{-}(a_{p_0})>\rho^{+}(z^{*})$. This, combined with $\rho^{+}(z^{*})\geq \rho^{-}(z^{*})$  gives $\rho^{-}(a_{p_0})> \rho^{-}(z^{*})$. By Lemma \ref{lemma existence of h} $\rho^{-}$ is strictly increasing on $[z^{*},H]$ and so $a_{p_0}>z^{*}$. We conclude that 

\begin{equation}\label{eq: rmk : boundary h: 03} 
z^{*}<\mathfrak{a}(s).
\end{equation}

Set $b_{n}=H-\frac{1}{n}$. By the left continuity of $\rho^{-}$(cfr Lemma \ref{le:boundonheight3} iv), $\left\lbrace \rho^{-}(b_{n})\right\rbrace^{\infty}_{n=1} $ converges to  $ \rho^{-}(H)$. This, with the fact that $s<\rho^{-}(H)$ yields
\begin{equation}\label{eq: rmk : boundary h: 04} 
s<\rho^{-}(b_{n})\leq \rho^{-}(H)
\end{equation}
for $n$ large enough. For such $n$, $b_n\in A(s)$ so that  $\mathfrak{a}(s)\leq b_{n}$. This, combined with $b_{n}< H$, yields 
\begin{equation}\label{eq: rmk : boundary h: 05}
 \mathfrak{a}(s)< H.
\end{equation}

From (\ref{eq: rmk : boundary h: 03}) and (\ref{eq: rmk : boundary h: 05}) we conclude that $\mathfrak{a}(s)\in (z^{*},H)$ which proves (ii). Thus, 
 there exists a sequence $ \left\lbrace z_{n}\right\rbrace^{\infty}_{n=1} $ in $(z^{*},H)$  such that $\mathfrak{a}(s)< z_{n}$ and  $\left\lbrace z_{n}\right\rbrace^{\infty}_{n=1}$ converges to $ \mathfrak{a}(s)$. As $(\mathfrak{a}(s),H]\subset A(s)$, we have that $z_n \in A(s)$ and so $\rho^-(z_n)\geq s$. Without loss of generality take $\left\lbrace z_{n}\right\rbrace_{n=1}^{\infty} $ to be  points of continuity of $\rho^{-}$ so that $\rho^{-}(z_n)=\rho^{+}(z_n)$. 
Therefore, as $\rho^{+}$ is right continuous, $\rho^{+}(\mathfrak{a}(s))=\lim_{n\rightarrow\infty} \rho^{+}(z_{n})=\lim_{n\rightarrow\infty} \rho^{-}(z_{n})\geq s$. On the other hand, let $\left\lbrace \bar{z}_{n}\right\rbrace^{\infty}_{n=1}$ be a sequence in $(z^{*},H)$ such that  $\left\lbrace \bar{z}_{n}\right\rbrace^{\infty}_{n=1}$ converges to $  \mathfrak{a}(s)$ and $\bar{z}_{n}< \mathfrak{a}(s)$ so that $\bar z_n\notin A(s)$. Then, necessarily  $\rho^{-}(\bar{z}_{n})<s$.
Hence $\rho^{-}(\mathfrak{a}(s))=\lim_{n\rightarrow\infty}\rho^{-}(\bar{z}_{n})\leq s$ by using the left continuity of $\rho^{-}$. We conclude that
\begin{equation}\label{eq: rmk : boundary l: 0000}
 s\in[\rho^{-}(\mathfrak{a}(s)),\rho^{+}(\mathfrak{a}(s))].
\end{equation}
Note that $[\rho^{-}(\mathfrak{a}(s)),\rho^{+}(\mathfrak{a}(s))]$ is an element of the family $\left\lbrace [\rho^{-}(z),\rho^{+}(z)]\right\rbrace _{z^*< z< H}$  in which elements are  disjoint from each other thanks to
 Lemma \ref{lem: existence of h inverse 1}(ii). If $s\in[\rho^{-}(z_0),\rho^{+}(z_0)]$ for some $z_0\in(z^*, H)$ then necessarily 
 $\rho^{-}(z_0)=\rho^{-}(\mathfrak{a}(s))$ and so  $z_0=\mathfrak{a}(s)$ in light of the fact that $\rho^-$ is strictly increasing on $[z^*,H]$. Note that $\mathbf{a}(s)= z^*$ for any $s\in [\rho^-(z^*),\rho^+(z^*)]$ by Remark \ref{rmk: inverse function on intervals} (3). This concludes the proof of (iii).
\endproof

\begin{corollary}\label{co: existence of h inverse 1}
 Assume the hypotheses in Lemma \ref{lem: inverse function on intervals} hold. The function $\mathfrak{a}$  is Lipschitz continuous.\\
\end{corollary}

 \proof{}
 We first note that as $\mathfrak{a}$ is non decreasing, we only need to show that 
\begin{equation}\label{l tilde is lip} 
\mathfrak{a}(s_{2})-\mathfrak{a}(s_{1})\leq c_0(s_{2}-s_{1} )
\end{equation} 
for all $s_{2}\geq s_{1}$ in $[0,\rho^{-}(H)]$ and some constant $c_0$.

(a)  Assume $\rho^{+}(z^{*})<s_{1}< \rho^-(H)$. Lemma \ref{lem: inverse function on intervals} (iii) ensures that $s_1\leq \rho^+(\mathfrak{a}(s_1))$ so that $\rho^{+}(z^{*})<\rho^+(\mathfrak{a}(s_1))$. As $\rho^+$ is strictly increasing on $[z^*,H]$ (see Lemma \ref{lem: existence of h inverse 1}), we obtain that $z^{*}< \mathfrak{a}(s_1)$. Let $ s_{2}\in [0, \rho^-(H)]$ such that $s_1<s_2$.  As $\mathfrak{a}$ is non decreasing, $\mathfrak{a}(s_1)\leq\mathfrak{a}(s_2)$. Thus, $z^*<\mathfrak{a}(s_1)\leq\mathfrak{a}(s_2).$ If $\mathfrak{a}(s_1)=\mathfrak{a}(s_2)$ then (\ref{l tilde is lip}) holds.
In the sequel, we assume $z^*<\mathfrak{a}(s_1)<\mathfrak{a}(s_2)$.
  Choose   $\bar{z}^{n}>\mathfrak{a}(s_1)$ such that $\left\lbrace \bar{z}^{n}\right\rbrace^{\infty}_{n=1} $ converges to $\mathfrak{a}(s_1)$ and  $\bar{z}^{n}$ are points of continuity of $\rho^{-}$,  that is, $\rho^{-}(\bar{z}^{n})=\rho^{+}(\bar{z}^{n})$. We use the fact  that $\rho^{+}$ is non decreasing to obtain  
$$\rho^{+}(\mathfrak{a}(s_{1}))\leq \rho^{+}(\bar{z}^{n})=\rho^{-}(\bar{z}^{n}).$$
 This, with the fact that $s_{1}\leq \rho^{+}(\mathfrak{a}(s_1))$ implies that $s_{1}\leq \rho^{-}(\bar{z}^{n})$ which we use  along with (\ref{eq:lip property of h}) and the fact that $\rho^{-}(\mathfrak{a}(s_2))\leq s_{2}$ (see Lemma \ref{lem: inverse function on intervals} (iii)) to get 
$$ \mathfrak{a}(s_2)-\bar{z}^{n}\leq c_0(\rho^{-}(\mathfrak{a}(s_2))-\rho^{-}(\bar{z}^{n}))\leq c_0(s_{2}-s_{1}).$$
By letting $n\rightarrow\infty$ we obtain (\ref{l tilde is lip}) for $\rho^{+}(z^{*})<s_{1}<s_2 \leq \rho^-(H)$.\\
By Lemma \ref{lem: inverse function on intervals} (iii),  $\mathfrak{a}(s)=z^{*}$ for all $s\in [ \rho^{-}(z^{*}), \rho^{+}(z^{*})].$ To show that $\mathfrak{a}$ is Lipschitz continuous on $[ \rho^{-}(z^{*}), \rho^{-}(H)]$, it suffices to show that  $\mathfrak{a}$ is continuous at $\rho^{+}(z^{*})$ and more precisely right continuous at  $\rho^{+}(z^{*})$.\\
Let $\rho^{+}(z^{*})\leq s_{n}\leq\rho^{-}(H)$ such that $\left\lbrace s_{n}\right\rbrace^{\infty}_{n=1} $ converges to  $ \rho^{+}(z^{\ast})$. By Lemma \ref{lem: inverse function on intervals} (iii), $s_{n}\in [\rho^-(\mathfrak{a}(s_n)),\rho^+(\mathfrak{a}(s_n)) ]$ so that  $\rho^-(\mathfrak{a}(s_n))$ converges to $\rho^{+}(z^{\ast})$. We use (\ref{eq:lip property of h}) to obtain that 
\begin{equation}\label{eq:lip property of h used1}
 0\leq\mathfrak{a}(s_n)-z^*\leq c_0(\rho^-(\mathfrak{a}(s_n))-\rho^+(z^*)).
\end{equation}

As $\rho^-(\mathfrak{a}(s_n))$ converges to $\rho^{+}(z^{\ast})$, (\ref{eq:lip property of h used1}) implies that $\mathfrak{a}(s_{n})$  converges to $z^{*}=\mathfrak{a}( \rho^{+}(z^{\ast}))$. We conclude that $\mathfrak{a}$ is continuous at $\rho^{+}(z^{\ast})$ and so Lipschitz on $[\rho^{-}(z^{*}), \rho^{-}(H)]$. In the case where $\rho^-(z^*)>0$, we have $z^*=0$ by definition of $z^*$. But since $0=z^*\leq \mathfrak{a}(s)\leq\mathfrak{a}(\rho^-(z^*))=z^*= 0$ for any $s\in[0,\rho^-(z^*)]$, we conclude that (\ref{l tilde is lip}) holds on $ [ 0, \rho^{-}(H)].$ \endproof\\
Set 
$$\mathfrak{D}_{\rho^-}=\left\lbrace (s,z): z^{*}\leq z\leq H, 0\leq s\leq \rho^{-}(z)\right\rbrace $$
and
$$Q=\left\lbrace (s,z): z^{*}< z< H, 0< s< \rho^{-}(H),\;\; z^*<\mathfrak{a}(s)<z \right\rbrace. $$

\begin{lemma}\label{lem: squeeze the domain D}
  Assume the hypotheses in Lemma \ref{lem: inverse function on intervals} hold. Then
 $Q\subset \mathfrak{D}_{\rho^-}\subset \bar{Q}$ and $Q$ is open.\\
\end{lemma}

\proof{} As $\mathfrak{a}$ is continuous, $Q$ is open.  $Q\subset \mathfrak{D}_{\rho^-} \subset \bar{Q}$ is straightforward.
\endproof

We observe that the boundary of $Q$ is the union of the following sets: $$Q_{1}=\left\lbrace (s,z): 0\leq s\leq \rho^{-}(H),\;\;z= \mathfrak{a}(s)\right\rbrace,\; Q_{2}=\left\lbrace (s,z): z=H,\; 0\leq s\leq  \rho^{-}(H) \right\rbrace $$
$$Q_{3}=\left\lbrace (s,z): s=0\;\;z^{*} \leq z \leq H \right\rbrace,\; Q_{4}=\left\lbrace (s,z): z=z^*,\; 0\leq s\leq \rho^{-}(z^*) \right\rbrace.  $$

\begin{prop}\label{lem: Lip boundary}
 Assume the hypotheses in Lemma \ref{lem: inverse function on intervals} hold. The boundary of the domain $D_{\rho^-}$ is  the union of the graphs of Lipschitz continuous functions.\\
\end{prop}

\proof{} Lemma  \ref{lem: squeeze the domain D} ensures that $Q$ and $\mathfrak{D}_{\rho^-}$ have the same boundary and clearly,
$$ D_{\rho^-}= \mathfrak{D}_{\rho^-}\cup \left\lbrace (s,z) : s=0,\; 0\leq z\leq z^*\right\rbrace. $$
The result follows immediately.\\

\endproof
\subsection{Existence and uniqueness of a solution for the Monge -Ampere equation}\hfill
\hfill

Here, we  prove the main theorem of the section.\\
\hfill

\proof{ of Theorem \ref{thm : existence and uniqueness MA}}
Proposition  \ref{Proposition MA} (iii) and (iv) show that  (\ref{eq:primal1})    has a unique minimizer $ \bar \rho$ such that $\bar \rho$ is monotone non decreasing and bounded away from $\frac{1}{2r_0^2}$,  and (\ref{eq:statement of dual problem}) has a unique maximizer $(\bar P,\bar\Psi )\in \calU_0$  so that (\ref{eq:Monge Ampere}) has a solution.  This variational solution is weak if $\sigma<<\calL^2$ and weak dual if  $\sigma<\nless\calL^2$. Proposition  \ref{Proposition MA} (iii) and then (ii)  guarantee that $I(\bar{\gamma},\bar \rho)=J[\sigma](\bar P,\bar{\Psi})$ where $\id\times\nabla \bar {P}$ pushes $e(s)\chi_{D_{\bar \rho}}$ to $\bar{\gamma}$. In view of Proposition  \ref{Proposition MA} (ii), we can assume without loss of generality that $\bar\rho (z)$ is the smallest value of $Argmin\Pi_{P}(\cdot,z)$ for all $z\in [0,H]$. As $e(s)\chi_{D_{\bar \rho}}$ is a probability measure and $\bar \rho$ monotone non decreasing, $\left\lbrace \bar \rho> 0 \right\rbrace $  is of positive Lebesgue measure so that $z*<H$ ($z^*$ is  as defined in Lemma \ref{lemma existence of h}). Note that if  $\spt(\sigma)\subset[\frac{1}{R_0},R_0]\times [0,R_0]$  then by (\ref{eq:subgradients}), $\bar P$ satisfies (\ref{eq:conditionP for regularity of the domain}). In this case, we  use Lemma \ref{lem: Lip boundary} to conclude that  $\partial D_{\bar \rho}$  is  the finite union of  graphs of Lipschitz continuous functions.
\endproof

\section{Some stability results}
\label{sec:Some stability results}

Theorem \ref{thm : existence and uniqueness MA} generates two operators  $\calH$,  $\bar\calH$ defined  in the following way:
To any $\sigma\in \mathcal{P}_{[R_0]}$ the operator $\calH$ associates   $\rho$, the minimizer  in  (\ref{eq:primal1}) and $\bar\calH$ associates the convex functions $(P,\Psi)\in \mathcal{U}_0$, the maximizer in (\ref{eq:statement of dual problem}).\\

We refer the reader to \cite{Sedjro} for the proofs of the following lemmas.

\begin{lemma}\label{lem: maximizers are precompact} 
 Let $\left\lbrace\sigma_n\right\rbrace^{\infty}_{n=1}$ and $\sigma$ be elements in $\mathcal{P}_{[R_0]}$  such that $\left\lbrace\sigma_n\right\rbrace^{\infty}_{n=1}$ converges to $\sigma$ narrowly. If $(P_n, \Psi_n)=\bar \calH (\sigma_n)$ then  $\left\lbrace P_n\right\rbrace^{\infty} _{n=1}$ and  $\left\lbrace \Psi_n\right\rbrace^{\infty} _{n=1}$ are precompact respectively in  $C(\bar{\mathcal{D}})$  and $C([0,R_0]^2)$.
\end{lemma}\\

 The next lemma uses the Helly theorem,  standard compactness results for optimal  plans and uniqueness results in theorem \ref{thm : existence and uniqueness MA}. \\

\begin{lemma}\label{lem :conv of the gradients}   
 Let  $\left\lbrace \sigma_n\right\rbrace^{\infty} _{n=1}$, $\sigma$ be elements in $\mathcal{P}_{[R_0]}$  and let $\rho_n=\calH(\sigma_n)$, $\rho=\calH(\sigma)$, $\bar\calH (\sigma)=(P,\Psi)$ and $\bar\calH (\sigma_n)=(P_n,\Psi_n)$  for $n\geq 1$. Assume that $\left\lbrace\sigma_n\right\rbrace^{\infty}_{n=1}$ converges narrowly to $\sigma$. Then 

(i)  $\left\lbrace \rho_n\right\rbrace^{\infty}_{n=1}$ converges pointwise to $\rho$  and so
$e(s)\chi_{D_{\rho_n}}$ converges narrowly to $e(s)\chi_{D_{ \rho}}.$ Moreover, if $\left\lbrace P_n\right\rbrace^{\infty}_{n=1}$ is uniformly convergent in $C(\bar{\mathcal{D}})$ then there exists $M>0$ such that 
\begin{equation}\label{eq :h_n and h are bounded} 
 2r_0^2M<1 \qquad \hbox{and} \qquad 0\leq\rho_n,  \rho <M \quad \text{ for } n\geq 1.
\end{equation}
(ii)   
\begin{equation}\label{eq : conv gradPn to gradP} 
 \nabla P_{n}\longmapsto \nabla P    \qquad  \mathcal{L}^2-a.e \text{  in  }  \mathcal{D}.
\end{equation}
Moreover, if, in addition,   $\left\lbrace \sigma_n\right\rbrace^{\infty} _{n=1}$, $\sigma$ are absolutely continuous with respect to the Lebesgue measure then
\begin{equation}\label{eq : conv gradPsin to gradPsi}
 \nabla \Psi_{n}\longmapsto \nabla \Psi    \qquad  \mathcal{L}^2-a.e \text{  in  }  \mathbb{R}_+^2.
\end{equation}
 \end{lemma}


\section{ Continuity equation for the forced Axisymmetric  Model}
\label{sec:Continuity equation for the Toy Model}
Our goal in this section is to solve the continuity equation (\ref{eq00: Continuity equation unphysical problem}) corresponding to the forced axisymmetric flow discussed in the introduction, under two different sets of assumptions on the initial data. Throughout this section, we assume that $g,\theta_0$ are positive constants and $R_0 >1$.


\subsection{ Existence of a solution  for initial data that are absolutely continuous with respect to the Lebesgue measure}

In this section, $\Sigma$ denotes the set of all Borel probability measures  $\sigma$ on $ \mathbb{R}^2_{+}$  that are absolutely continuous with respect to the Lebesgue measure and whose support is contained in $[0,R_0]^2$
 
We consider the functions  $ F= F_t(r,z),  S= S_t(r,z)$   such that  $S,  F\in C^{1}((0,\infty)\times\mathbb{R}^2)$ and satisfy the following conditions:

\begin{itemize}
 \item  ($A1$)  $0\leq  F, \frac{g}{\theta_0} S \leq M  $  for some positive constant $ M $.
 \item ($A2$)  $ F=F(r)$ and $ S=S(z).$
 \item  ($A3$) $\frac{\partial  F}{\partial r},  \frac{\partial  S}{\partial z}>0. $
\end{itemize}

The next lemma is a well known result and can be found in \cite{Sedjro}.\\

\begin{lemma}\label{lem: convergence in distribution when measures abs continuous} 
We consider a family  $\sigma=\vrho\calL^2,\; \sigma^{n}=\vrho^{n}\calL^2\in \mathcal{P}( \mathbb{R}^2)\cap L^1( \mathbb{R}^2) $ $n\geq 1$ that is equi-integrable and let $\left\lbrace {\bf v}^{n}\right\rbrace_{n\geq 1}: \mathbb{R}^2\longmapsto \mathbb{R}^2$ be Borel measurable such that $|{\bf v}^{n}|\leq M_0\quad a.e$  where $M_0 $ is a positive constant.
 Assume $\left\lbrace \sigma^{n}\right\rbrace^{\infty}_{n=1}$ converges narrowly to $\sigma$ and ${\bf v}^{n}$ converges to ${\bf v}\quad a.e $. Then 
$${\bf v}^{n} \sigma^{n}\longrightarrow  {\bf v}\sigma \qquad \text{ in the sense of distributions.}$$
\end{lemma}

\begin{lemma}\label{lem unphisacal1: continuity equation on time step} 
Let $a, \;\tau>0$ and $L_a>1$. Let $\sigma_a=\vrho_a\calL^2\in \mathcal{P}_{[L_0]}$  be a Borel probability measure that is absolutely continuous with respect to the Lebesgue measure. Assume  that $\Psi(a,\cdot) :\mathbb{R}_{+}^{2}\longmapsto \mathbb{R}$ is convex and  that whenever $\nabla \Psi(a,\cdot)$  exists, it has values in $[0,\frac{1}{2r_0^2})\times[0, H]$.
Set 
\begin{equation}\label{eq unphysical1: define velocity v} 
 {\bf v}_t (q)= {\textstyle \left(\;2\sqrt{\Upsilon}  F_t\left(\frac{r_0}{\sqrt{1-2r_0^2\frac{\partial \Psi}{\partial \Upsilon}(a,q)}},\frac{\partial\Psi}{\partial Z} (a,q)\right)  ,\;  \frac{g}{\theta_0} S_t\left( \frac{r_0}{\sqrt{1-2r_0^2\frac{\partial \Psi}{\partial \Upsilon}(a,q)}},\frac{\partial\Psi}{\partial Z}(a,q) \right)  \right) }
\end{equation}
 with $q=(\Upsilon, Z)$. Assume that $(A_1)$,$(A_2) $and $(A_3)$ hold. Then, there exists  a family of measures $\sigma_t=\vrho_t\mathcal{L}^2\in \mathcal{P}(\mathbb{R}^{2})$  absolutely continuous with respect to Lebesgue such that
$$spt\sigma_t\subset[0, L_{a+\tau}]^2\qquad   \text{  for  } t\in[a,a+\tau] \text{  with  }  1< L_{a+\tau}:=L_a(1+ M\tau)^2 $$
satisfying the following:\\
(a) $\textstyle {\int}_{\mathbb{R}^2} \vrho_t^r dq\leq \textstyle {\int}_{\mathbb{R}^2} \vrho_a^{r}dq $  for any  $ r\geq 1 $ and $t\in [a, a+\tau]$.\\
(b) $t\longmapsto\sigma_t\in AC_{1}\left( (a,a+\tau);\mathcal{P}( \mathbb{R}^2)\right) $ and 

\begin{equation}\label{eq unphysical1:continuity equation with constant vector} 
\begin{cases}
 \frac{\partial\sigma}{\partial t}+ \div ({\sigma\bf v}_t)=0, \qquad  (t,q)\in (a,a+\tau)\times\mathbb{R}^2 \\
\sigma_{|t=a}=\bar\sigma_a
\end{cases}
\end{equation}
holds in the sense of distributions.\\
(c) $t\longmapsto\sigma_t$ is Lipschitz continuous with respect to the $1-$Wasserstein distance with Lipschitz constant less than $c_0:=M\sqrt{4L_0+1}$ in $[a,a+\tau] $ .
\end{lemma}
\\
\begin{rmk}
Since $\Psi(a,\cdot)$ is convex, $\nabla \Psi(a,\cdot)$ exists $\calL^2$ $a.e$ so that ${\bf v}_t$ is defined $\calL^2$ $a.e $. As  $\sigma_a$  is absolutely continuous with respect to $\calL^2$,  ${\bf v}_t$ is defined $a.e$ $\sigma_a.$\\
\end{rmk}

\proof{} Set $U=(0,\infty)\times(0,\infty)$. We divide the proof into a several steps.\\
\textbf{Step 1}. We assume that $\Psi(a,\cdot)$ is  $C^2(U)$.\\
We observe that  the vector field ${\bf v}$ is smooth in $ (0,\infty)\times U$ and  define the associated flow by 
\begin{equation}\label{eq unphysical1: defining flow} 
 \dot{\phi}_t={\bf v}_t\circ \phi_t \text{ and } \phi_a= \id \qquad \text{   for   }  t\in (a,a+\tau). 
\end{equation}

We note that $\sigma_t=\phi_{t\#}\sigma_a$  solves the continuity equation (\ref{eq unphysical1:continuity equation with constant vector}). In view of $(A2)$,  A simple computation  gives

$$\div\ [{\bf v}_t]=\frac{1}{\sqrt{\Upsilon}}F+ \frac{r^3_0\sqrt{\Upsilon}}{\Big(\sqrt{1-2r_0^2\frac{\partial \Psi}{\partial \Upsilon}}\Big)^3} \frac{\partial^{2}\Psi}{\partial\Upsilon^2}\dfrac{\partial F}{\partial r}+ \frac{\partial^{2}\Psi}{\partial Z^2}\dfrac{\partial S}{\partial z}.$$
 Since $\Psi(a,\cdot)$ is convex, its second partial derivatives with respect to $\Upsilon$ and $Z$ are all non negative. This, combined with $(A3)$ leads to
             $$\div\ [{\bf v}_t]\geq 0,$$
which ensures that  $t\longmapsto \det(\nabla\phi_t)$ is non decreasing and so, 
\begin{equation}\label{eq unphysical1:det of the gradient >1} 
 \det(\nabla\phi_t)\geq \det(\nabla\phi_a)=1.
\end{equation}
\textbf{Step 2}. We use $(A1)$, the fact that $L_a>1$ and the definition of the flow in (\ref{eq unphysical1: defining flow}) to establish  a bound on the range of $\phi_t$ for $t$ fixed :

$$\phi_t([0,L_a]^2)\subset [0, L_a(1+ M (t-a))^2]^2. $$
Therefore, as $\sigma_t=\phi_{t\#}\sigma_a$ and $\phi_t$ is continuous,\\
$$ spt(\sigma_t) =\overline{\phi_t\left( spt(\sigma_a) \right)} \subset\phi_t([0,L_a]^2)\subset [0, L_a(1+ M(t-a))^2]^2.  $$
\textbf{Step 3}.  In view of (\ref{eq unphysical1:det of the gradient >1}),
 $\sigma_t=\phi_{t\#}\sigma_{a}$ is absolutely continuous with respect to  the Lebesgue measure $ \mathcal{L}^2$ and its density function $\vrho_{t}$ satisfies
\begin{equation}\label{eq2 unphysical1:det of the gradient >1} 
 \vrho_t\circ\phi_t=\frac{\vrho_a}{\det(\nabla\phi_t)}\leq \vrho_a.
\end{equation}
Using (\ref{eq2 unphysical1:det of the gradient >1}) and the fact that $\det [\nabla \phi_t]\geq 1$, we obtain 
$$\int_{\mathbb{R}^2}\vrho_t^r dq \leq  \int_{\mathbb{R}^2}  \vrho_a^r\circ \phi^{-1}\det [\nabla \phi]^{-1}\circ\phi^{-1} dq= \int_{\mathbb{R}^2}  \vrho_a^r dq     \qquad r\geq 1.$$
 This establishes (a). We easily check  $|{\bf v}|\leq M\sqrt{4L_a+1}= c_0$ and so, by [ Theorem 8.3.1 in\cite{ags:book}],
\begin{equation}\label{eq unphysical1:lip inequality} 
 W_1(\sigma_t, \sigma_s)\leq \int_s^t ||{\bf v}_{r}||_{L^1(\sigma_r )}dr \leq c_0 (t-s) \qquad  \text{for all } a\leq s \leq t \leq a+\tau
\end{equation}
 Therefore, $t\longrightarrow\sigma_t$ is $c_0$-Lipschitz continuous on $[a,a+\tau].$ Thus,
, \begin{equation}\label{eq unphysical1: equi bounded in W_1} 
   W_1(\bar\sigma_a, \sigma_t)\leq c_0(t-a)\leq c_0\tau
  \end{equation}
for all $t\in [a,a+\tau]$. As a consequence, $\left\lbrace \sigma_t\right\rbrace _{t\in [a,a+\tau]}$ is  bounded in the $1-$Wasserstein space.

\textbf{Step 4}.
We consider now the general case where $\Psi(a,\cdot)$ is not necessary smooth. We note that, as $\Psi(a,\cdot)$ is convex, $\Psi(a,\cdot)$  is locally Lipschitz and so, $\Psi(a,\cdot)\in W^{1,1}_{loc}(U)$. We set 
$$  U_n=\left\lbrace x\in U: dist(x,\partial U)> n^{-1} \right\rbrace, \quad \Psi^{n} (a,\cdot):= \Psi(a,\cdot)\ast j_n  \quad \text{on } U_n. $$
  Here, $\partial U$ denotes the boundary of $U$ and $j_n(x)=n^{2}j(nx)$ where $j$ is the standard mollifier. We obtain that $ \Psi^{n} (a,\cdot)$ converges  to $ \Psi (a,\cdot)$ in $ W^{1,1}_{loc}(U)$. This convergence guarantees that, up to a subsequence, $\nabla \Psi^n(a,\cdot)$ converges to $\nabla \Psi(a,\cdot)$,  $a.e \text{ in }U$.\\
Let us denote by ${\bf v}^n$ the velocity field when $\Psi$ is replaced by $\Psi^n$ in (\ref{eq unphysical1: define velocity v}). Without loss of generality,  we have that 
$${\bf v}_t^n\longrightarrow {\bf v}_t\qquad a.e.$$
Let $\sigma^{n}=\vrho^n\mathcal{L}^2$ denote the solution of (\ref{eq unphysical1:continuity equation with constant vector}) when ${\bf v}$ is replaced by ${\bf v}^n$. Then, $\sigma^{n}$ satisfies 
(\ref{eq unphysical1:lip inequality}) and the conditions (a), (b) and (c). We  obtain that the curve $t\longrightarrow\sigma_t^n$ is equi-Lipschitz on $[a,a+\tau]$ with respect to $W_1$ and (\ref{eq unphysical1: equi bounded in W_1}) ensures that it  is equi-bounded in $\mathcal{P}(\mathbb{R}^2)$ with respect to $W_1$. Therefore, there exists a subsequence that we still denote by $t\longrightarrow\sigma^{n}_t$ ( $n$ is independent of $t$) such that $\left\lbrace \sigma_t^n\right\rbrace^{\infty}_{n=1} $ converge narrowly to $\sigma_t$ for each $t\in[0,\tau]$.\\
Since the Wasserstein distance is lower semi-continuous with respect to narrow convergence and $\sigma_t^n$  satisfy (\ref{eq unphysical1:lip inequality}), $\sigma_t$ also satisfies (\ref{eq unphysical1:lip inequality}). That is, $\sigma_t$  is $c_0$-Lipschitz continuous on $(a,a+\tau)$, which proves (c). By condition (a),  $\left\lbrace \vrho_t^n\right\rbrace_{n=1}^{\infty}$ is equibounded in $L^r$, $r\geq 1$. Consequently, as $\left\lbrace\vrho_t^n\right\rbrace_{n=1}^{\infty}$ converges weakly* to $\sigma_t$, the Dunford-Pettis theorem guarantees that $\sigma_t$ is absolutely continuous with respect to Lebesgue , that is, $\sigma_t= \vrho_t\mathcal{L}^2$  . Also, as $\left\lbrace\vrho_t^n\right\rbrace_{n=1}^{\infty}$ satisfy the condition (a), the weakly lower semi-contnuity of the $L^r$ norms ensures that $\vrho_t$  satisfies the condition (a) as well. 

To obtain the continuity equation in (b), we only need to show that $\left\lbrace {\bf v}_t^n\sigma_t^n\right\rbrace_{n=1}^{\infty}$  converges to ${\bf v}_t\sigma_t$ in the sense of distributions for each $t$ fixed, as the fact that $\left\lbrace {\bf v}_t^n\sigma_t^ndt\right\rbrace_{n=1}^{\infty}$  converges to ${\bf v}_t\sigma_tdt$ in the sense of distributions will be obtained by a simple application of Lebesgue dominated convergence theorem.  We note that the inequality in (a) ensures that $\left\lbrace\vrho_t^n\right\rbrace_{n=1}^{\infty}$  is  equi- integrable. As $ {\bf v}^n_t$ converges to ${\bf v}_t$ a.e and $\sigma^n=\vrho^n \calL^2 $ narrowly to $\sigma =\vrho \calL^2$, we use Lemma \ref{lem: convergence in distribution when measures abs continuous} to obtain  the desired result. \endproof \\
If $\sigma\in \Sigma$ and  $(P,\Psi)=\bar\calH(\sigma)$ then   we  define 

\begin{equation}\label{eq2 unphysical1: velocity in dual space} 
 V_t[\sigma]:= {\textstyle \left( \;2\sqrt{\Upsilon}\left[  F_t\left(\frac{r_0}{\sqrt{1-2r_0^2\frac{\partial \Psi}{\partial \Upsilon}}},\frac{\partial\Psi}{\partial Z} \right) \right]  ,\; \frac{g}{\theta_0}  S_t\left(\frac{r_0}{\sqrt{1-2r_0^2\frac{\partial \Psi}{\partial \Upsilon}}},\frac{\partial\Psi}{\partial Z} \right)  \right). }
\end{equation}

\begin{thm}
 Assume that (A1),(A2) and (A3) hold. Assume  $1<L_0 < R_0$ and let $\bar\sigma_0=\bar\vrho_0\mathcal{L}^2\in\Sigma$ such that 
$$\spt (\bar\sigma_{0}) \subset  [0,L_0]^2.$$
Let $T>0$  such that $ L_0 e^{6MT}\leq R_0$. Then, there exists $\sigma_t=\vrho_t\mathcal{L}^2\in\Sigma$ satisfying :\\
(a) $\int_{\mathbb{R}^2} \vrho_t^r dq\leq \int_{\mathbb{R}^2} \bar\vrho_0 ^r dq$  for any  $r\geq 1. $\\
(b) $ t\longmapsto\sigma_t\in AC_{1}\left( (0,T);\mathcal{P}_2( \mathbb{R}^2)\right) $ and 

\begin{equation}\label{eq unphysical1:continuity equation} 
\begin{cases}
 \frac{\partial\sigma}{\partial t}+\div (\sigma V_t[\sigma])=0, \qquad  (0,T)\times \mathbb{R}^2 \\
\sigma_{|t=0}=\bar\sigma_0
\end{cases}
\end{equation}
holds in the sense of distributions.\\
(c) $t\longmapsto\sigma_t$ is Lipschitz continuous with Lipschitz constant less than $c_0=M\sqrt{4L_0+1}.$\\
\end{thm}

\proof{} We divide the proof into 2 steps.\\
\textbf{Step 1} We fix a non negative integer $N$ and divide the interval $[0,T]$ into $N$ intervals with equal length $\tau=\frac{T}{N}.$ We first show that we can construct  a discrete function $\sigma_t^N=\vrho^{N}_t\calL^2$ satisfying the following properties:\\
(a1)  $\int_{\mathbb{R}^2} (\vrho^{N}_t)^r dq \leq \int_{\mathbb{R}^2} (\vrho_0 )^r dq $  for any  $r\geq 1$.\\
(b1) The ``delayed'' equation 

\begin{equation}\label{eq unphysical2:continuity equation} 
\begin{cases}
 \frac{\partial\sigma^{N}_t}{\partial t}+ \div ({\sigma^{N}_t\bf v}^N_t)=0, \qquad  (0,T)\times\mathbb{R}^2 \\
\sigma_{|t=0}=\bar\sigma_0
\end{cases}
\end{equation}
holds in the sense of distributions with ${\bf v}_t^N=V_{t}[\sigma^{N}_{[\frac{t}{\tau}]\tau}] $  for all $t\in [0,T).$\\
(c1) $t\longmapsto\sigma^{N}_t$ is Lipschitz continuous with respect to $W_1$ with Lipschitz constant less than $c_0.$\\

The construction of $\sigma_t^{N}$ goes as follows: we start off by setting $\sigma^{N}_0=\bar\sigma_0$ and  ${\bf v}_t^N=V_t[\bar\sigma_0]$ for $t\in[0, \tau]$. We use Lemma \ref{lem unphisacal1: continuity equation on time step}  to obtain a solution $\sigma_t^{N}$  on $[0, \tau]$. We repeat inductively the same process $(N-1)$ times by setting $\sigma_{i\tau}:=\sigma^N_{i\tau}$ and ${\bf v}^N_t=V_{t}[\sigma_{i\tau}]$ for $t\in [i\tau, (i-1) \tau]$ and using Lemma \ref{lem unphisacal1: continuity equation on time step} to obtain $\sigma_t^{N}$ on $ [i\tau, (i+1) \tau]$. In view of Lemma \ref{lem unphisacal1: continuity equation on time step}, we note that the process described above works as long as  $\left\lbrace \sigma_{i\tau}\right\rbrace_{1\leq i\leq N} $ stays absolutely continuous with respect to  the Lebesgue measure and is compactly supported in $\mathbb{R}^2_{+}.$ We next show that $\left\lbrace \sigma_{i\tau}\right\rbrace_{1\leq i\leq N} \subset \Sigma.$ We first observe that, by construction, Lemma \ref{lem unphisacal1: continuity equation on time step} guarantees that  $\left\lbrace \sigma_{i\tau}\right\rbrace_{1\leq i\leq N}$  are absolutely continuous with respect to the Lebesgue measure in $\mathbb{R}_{+}^2$. Define
 $$L_i:= \max \left( \sup\left\lbrace \Upsilon : ( \Upsilon, Z) \in  \spt (\sigma_{i\tau}) \right\rbrace;  \sup\left\lbrace Z: (\Upsilon, Z) \in  \spt (\sigma_{i\tau})  \right\rbrace \right)  $$
 for $1\leq i\leq N$.
By Lemma \ref{lem unphisacal1: continuity equation on time step},
$$ L_{i+1}\leq L_i( 1+M\tau)^2 \leq L_0( 1+M\tau)^{2(i+2)} < L_0( 1+M\tau)^{6N}= L_0( 1+M\frac{T}{N})^{6N}\leq  L_0 e^{6MT} .$$

With the constraint   $ L_0 e^{6MT}<R_0$ on $T$,  we  obtain that for any $0\leq i\leq N$,  $\spt (\sigma_{i\tau}) $ is contained in $ [0,R_0]^2.$ Therefore, the above construction of $\sigma_{t}^{N}$ is thoroughly justified.
 We  easily check that the conditions (a1) and (c1) follow from the condition (a) and (c) of Lemma \ref{lem unphisacal1: continuity equation on time step}.

\textbf{Step 2}
 By  (c1), $t\longmapsto \sigma^N_{t}$ are  equi-Lipschitz continuous on $[0,T]$ and, since $ \sigma^N_{0}=\bar\sigma_0$ for all  $N$,  they are equibounded in the $1$-Wasserstein space. Thus, there exists a subsequence of $t \longmapsto \sigma_t^{N}$, still denoted by $t \longmapsto \sigma_t^{N}$ ($N$ independent of $t$),  such that
$\left\lbrace \sigma_t^{N}\right\rbrace^{\infty}_{N=1}$  converges narrowly to $\sigma_t$  for any $t\in[0,T]$.\\
 In view of (a1), the theorem of Dunford-Pettis ensures that $\sigma_{t}=\rho_t\mathcal{L}^2.$ The weakly lower semi -continuity of the $L^r$- norms leads to (a).
We next show that $\sigma_t$ satisfies (\ref{eq unphysical1:continuity equation}). As $\left\lbrace \sigma_t^{N}\right\rbrace^{\infty}_{N=1}$  converges narrowly to $\sigma_t$, we only need to show that $\left\lbrace {\bf v}_t^N\sigma_t^N\right\rbrace_{N=1}^{\infty}$  converges to  ${V}_t[{\sigma}_t]\sigma_t$ in the sense of distributions for each $t$ fixed, since the fact that  $\left\lbrace {\bf v}_t^N\sigma_t^N dt\right\rbrace_{n=1}^{\infty}$  converges to ${\bf v}_t[\sigma_t]\sigma_t dt $   in the sense of distributions will be obtained by a simple application of Lebesgue dominated convergence theorem. By (c1),
$$ W_1\left(\sigma^{N}_t,\sigma_{[\frac{t}{\tau}]\tau}^{N} \right)\leq | t-[\frac{t}{\tau}]\tau| \leq \frac{T}{N}. $$ 

In light of this estimate,  $\left\lbrace \sigma_t^{N}\right\rbrace^{\infty}_{N=1}$  converges narrowly to $\sigma_t$ implies that 
 $\left\lbrace\sigma_{[\frac{t}{\tau}]\tau}^{N}\right\rbrace^{\infty}_{N=1}$  converges narrowly to $\sigma_t$. Thus, for each $t$ fixed, $\left\lbrace {\bf v}_t^{N}\right\rbrace^{\infty}_{i=1}$ converges to $V_t[\sigma_t]$ $\mathcal{L}^2-a.e$ by Lemma \ref{lem :conv of the gradients}(ii). We use Lemma \ref{lem: convergence in distribution when measures abs continuous}  to obtain that  $\left\lbrace {\bf v}_t^N\sigma_t^N\right\rbrace_{N=1}^{\infty}$  converges to  ${V}_t[{\sigma}_t]\sigma_t$ in the sense of distributions for each $t$ fixed. This concludes the proof. \endproof


\subsection{ Existence of a solution  for general initial data}

In this section, we impose  the following conditions on the forcing terms $F$ and $S : \mathbb{R}_{+}\times \mathbb{R}^2\longrightarrow \mathbb{R}$ :  

\begin{itemize}
 \item  $(B_1)$ $F$ and $S$ are continuous and bounded.
 \item $(B_2)$ $F\geq 0$ and $S \geq 0$.
\end{itemize}

Set $$\mathbb{F}_t= \left(F_t\circ {\bf d}, \frac{g}{\theta_0}S_t\circ {\bf d} \right) \quad \hbox{with}\quad {\bf d}(s,z)=\left( \frac{r_0}{\sqrt{1-2r_0^2s}}, z\right).$$

As $ F$ and $S$ are bounded, there exists  a constant $ C_0$ (independent of $t$) such that

 $$\sqrt{R_0}||\mathbb{F}_t||_{\infty}\leq  C_0$$
for all $t\geq 0.$
To any function $G=(G_1,G_2)$, we associate $\calA [G]$ defined by  $$\calA [G](\Upsilon, Z)=\left( \sqrt{\Upsilon} G_1(\Upsilon,Z), G_2(\Upsilon,Z)\right). $$

Note that if $G\in C([0,R_0]^2; \mathbb{R}^2)$ then 

$$\calA [G]\in C([0,R_0]^2; \mathbb{R}^2) \qquad\hbox{ with } \qquad \|\calA [G]\|_{\infty} \leq \sqrt{R_0} \| G\|_{\infty}.$$
For $\sigma\in \mathcal{P}_{[R_0]}$, if $\rho=\calH(\sigma)$    and $(P,\Psi)=\bar\calH(\sigma)$ then we define 
$$L_t[\sigma](G):=\int_{\mathbb{R}^2}\langle \calA [G]\circ \nabla P, \;\mathbb{F}_t \rangle e(s)\chi_{D_{\rho}}(s,z)dsdz $$
for all $G$ Borel measurable. Note that if $G\in L^1(\sigma,\mathbb{R}^2)$ then 

\begin{equation}\label{eq: control calA(G)}
\calA [G]\in L^1(\sigma; \mathbb{R}^2)\qquad\hbox{ with } \qquad \|\calA [G]\|_{L^1(\sigma; \mathbb{R}^2)} \leq \sqrt{R_0} \| G\|_{L^1(\sigma; \mathbb{R}^2)}.
\end{equation}

\begin{rmk}\label{rmk:positive functional} For $G\in L^1(\sigma,\mathbb{R}^2)$ such that $G_1\geq 0,G_2\geq 0\quad \sigma\; a.e$ we have that $L_t[\sigma](G)\geq 0$.\\
\end{rmk}

\begin{lemma}\label{lem: Riesz for V[sigma]} 
 Fix $t>0$. There exists $V_t:\mathcal{P}_{[R_0]}\longrightarrow \displaystyle{\bigcup_{\sigma\in\mathcal{P}_{[R_0]}}}L^{\infty}(\sigma; \mathbb{R}^2)$ such that for any $\sigma\in \mathcal{P}_{[R_0]}$, $V_t[\sigma]:=\left(V^1_t[\sigma],V^2_t[\sigma] \right) \in L^{\infty}(\sigma; \mathbb{R}^2)$ and
\begin{equation}\label{eq: Riesz for V[sigma] 1} 
 L_t[\sigma](G)=\int_{\mathbb{R}^2}\langle V_t[\sigma],G\rangle d\sigma
\end{equation}
for all $G\in L^1(\sigma,\mathbb{R}^2)$. Moreover, 
\begin{equation}\label{eq: Riesz for V[sigma] 2} 
|| V_t[\sigma]||_{L^{\infty}(\sigma; \mathbb{R}^2)}\leq C_0
\end{equation}
and 
\begin{equation}\label{eq: Riesz for V[sigma] 3} 
V^1_t[\sigma] \geq 0 \qquad V^2_t[\sigma]\geq 0 \quad \sigma \text{ a.e. }
\end{equation}
\end{lemma}
\proof{} Let $\sigma\in \mathcal{P}_{[R_0]}$ and set  $\rho=\calH(\sigma)$    and $(P,\Psi)=\bar\calH(\sigma)$. We use (\ref{eq: control calA(G)}) to obtain
$$
\left\vert L_t[\sigma] (G) \right\vert \leq \|\mathbb{F}_t\|_{\infty}\int_{\mathbb{R}^2}\left\vert \calA [G]\right\vert \circ \nabla P\;\; e(s)\chi_{D_{\rho}}dsdz =\|\mathbb{F}_t\|_{\infty}\sqrt{R_0} \|G\|_{L^1(\sigma;\mathbb{R}^2)} \leq  C_0\|G\|_{L^1(\sigma;\mathbb{R}^2)}.\\
     $$
The Riesz representation theorem provides $V_t[\sigma]$ such that (\ref{eq: Riesz for V[sigma] 1}) and (\ref{eq: Riesz for V[sigma] 2}) holds.
Note that

\begin{equation}\label{eq: positive functional}
L_t[\sigma](G_1,0)=\int_{\mathbb{R}^2}\langle V_t[\sigma],(G_1,0)\rangle d\sigma=\int_{V^1_t[\sigma]\geq 0} V^1_t[\sigma]G_1 d\sigma +\int_{V^1_t[\sigma]< 0} V^1_t[\sigma]G_1 d\sigma. 
\end{equation}

Choose  $G_1=\chi_{\left\lbrace V^1_t[\sigma]< 0\right\rbrace }\geq 0$ so that $  L_t[\sigma](G_1,0) \geq 0$ (see remark  \ref{rmk:positive functional} ). If, in addition, we have $\sigma(V^1_t[\sigma]<0)>0$ then we use (\ref{eq: positive functional}) to obtain
$$ 0\leq L_t[\sigma](G_1,0)=\int_{\left\lbrace V^1_t[\sigma]< 0\right\rbrace } V^1_t[\sigma]d\sigma<0. $$
 Therefore,  $\sigma(V^1_t[\sigma]<0)=0$ so that $V^1_t[\sigma]\geq 0$ $\sigma$ a.e. A similar argument shows that $V^2_t[\sigma]\geq 0$ $\sigma$ a.e.
\endproof

\begin{rmk}
 If $\sigma\in \mathcal{P}_{[R_0]}$,  $\rho=\calH(\sigma)$ and $(P,\Psi)=\bar\calH(\sigma)$ then, for $G\in L^1(\sigma,\mathbb{R}^2)$ and any $t,r\geq 0$,
$$ 
 L_t[\sigma](G)- L_r[\sigma](G)= \int_{\mathbb{R}^2}\langle \calA [G]\circ \nabla P, \;\mathbb{F}_t -\mathbb{F}_r \rangle e(s)\chi_{D_h}(s,z)dsdz.
$$
And so, in view of Lemma \ref{lem: Riesz for V[sigma]},
 
\begin{equation}\label{eq control variation in time of velocity}
\begin{aligned}
\left\vert \int_{\mathbb{R}^2}\langle V_t[\sigma]-V_r[\sigma] ,G\rangle d\sigma \right\vert &=\left\vert L_t[\sigma](G)- L_r[\sigma](G) \right\vert\\
 &\leq \|\calA [G]\|_{\infty}　　\sup_{p\in \Delta_{r_0}}\left\vert\mathbb{F}_t(p) -\mathbb{F}_r (p)\right\vert　\int_{\mathbb{R}^2} e(s)\chi_{D_{\rho}}(s,z)dsdz \\
&\leq \sqrt{R_0}\|G\|_{\infty}　　\sup_{p\in \Delta_{r_0}}\left\vert\mathbb{F}_t(p) -\mathbb{F}_r (p)\right\vert.
\end{aligned}
\end{equation}
\end{rmk}

\begin{lemma}\label{lem: general measure: conv of momentum} 
Let $t\geq 0$ and $V_t$ as provided by Lemma \ref{lem: Riesz for V[sigma]}. Let $\left\lbrace\sigma_n \right\rbrace^{\infty}_{n=1}$ and $\sigma$ be elements of $\mathcal{P}_{[R_0]}(\mathbb{R}^2)$ such that $\left\lbrace\sigma_n \right\rbrace^{\infty}_{n=1}$ converges narrowly to $\sigma$. Then $\left\lbrace V_t[\sigma_n]\sigma_n\right\rbrace_{n}$ converges to $V_t[\sigma]\sigma$ in the sense of distributions.\\
\end{lemma}

\proof{}
Let $(P_n,\Psi_n)=\bar\calH(\sigma_n)$, $(P,\Psi)=\bar\calH(\sigma)$, $\rho_n=\calH(\sigma_n)$ and $\rho=\calH(\sigma)$. As $\left\lbrace\sigma_n \right\rbrace^{\infty}_{n=1}$ converges narrowly to $\sigma$, Lemma \ref{lem: maximizers are precompact}  ensures that there exists a subsequence $\left\lbrace n_k\right\rbrace^{\infty}_{k=1}  $ of integers such that $\left\lbrace P_{n_k}\right\rbrace^{\infty}_{k=1} $  converges uniformly. Hence, by Lemma \ref{lem :conv of the gradients} (i), $ 0\leq \rho, \rho_{n_k}\leq M_0<\frac{1}{2r_0^2}$  for some constant $M_0$ and so $\left\lbrace   e(s)\chi_{D_{\rho_{n_k}}} \right\rbrace^{\infty}_{k=1} $  is equi-integrable. Lemma \ref{lem :conv of the gradients}(ii) ensures that $\left\lbrace \nabla P_{n_k} \right\rbrace^{\infty}_{k=1} $ converges  a.e to $\nabla P$.  Let $G\in C([0,R_0]^2) $; then $\calA[G]$ is continuous on $[0,R_0]^2$ and  $\langle \calA[G]\circ \nabla P_{n_k}; \mathbb{F}_t\rangle$  converges a.e to $\langle \calA[G]\circ \nabla P ; \mathbb{F}_t\rangle$.  Moreover, as 
$ G $ is  bounded function, $\calA[G]$ is bounded. In addition, since  $\mathbb{F}$ is bounded,  there exists $M >0$ such that $\left\vert\langle \calA[G]\circ \nabla P_{n_k}; \mathbb{F}_t\rangle \right\vert \leq M$ for all $k\geq 1$ and $t>0$. Using Lemma \ref{lem: convergence in distribution when measures abs continuous}, we obtain that 
$$\lim_{k\rightarrow\infty}\int \langle \calA[G]\circ \nabla P_{n_k}; \mathbb{F}\rangle e(s)\chi_{D_{\rho_{n_k}}} (s,z)dsdz= \int \langle \calA[G]\circ \nabla P; \mathbb{F}\rangle e(s)\chi_{D_{\rho}} (s,z)ds dz.$$
This, in the light of (\ref{eq: Riesz for V[sigma] 1}), becomes 
\begin{equation}\label{eq: continuity of sigmaV}
\lim_{k\rightarrow\infty}\int_{\mathbb{R}^2}\langle G, V_t[\sigma_{n_k}]\rangle d\sigma_{n_k}= \int_{\mathbb{R}^2}\langle G,V_t[\sigma]\rangle d\sigma.
\end{equation}

As $\left\lbrace\sigma_n \right\rbrace^{\infty}_{n=1}$ and $\sigma\in\mathcal{P}_{[R_0]}(\mathbb{R}^2)$, (\ref{eq: continuity of sigmaV}) still holds for  $G\in C_c(\mathbb{R}^2)$. Thus, we obtain that  $\left\lbrace V_t[\sigma_{n_k}]\sigma_{n_k}\right\rbrace_{k}$ converges to $V_t[\sigma]\sigma$ in the sense of distributions. Since  the limit $V_t[\sigma]\sigma$  is independent of the  extracted subsequence of $\left\lbrace V_t[\sigma_n]\sigma_n\right\rbrace_{n}$, we conclude that the whole sequence $\left\lbrace V_t[\sigma_n]\sigma_n\right\rbrace_{n}$ converges narrowly to $V_t[\sigma]\sigma$. \endproof

\begin{definition}\label{eq1 unphysical1: def in the weak dual sense}
Let $T>0$ and $t\longrightarrow\sigma_t$ be an absolutely continuous path in $\mathcal{P}_{[R_0]}$. 

Let $(P(t,\cdot),\Psi(t,\cdot))= \bar\calH(\sigma_t)$ and $\rho_t= \calH(\sigma_t)$.  We say that 
$$ \dot{\sigma_t}=\chi^{\bar\calH}_{\calH}(\sigma_t)\qquad\qquad t\in (0,T)$$
 in the weak dual sense if  
\begin{equation}\label{eq: equation of continuity  weak dual sense} 
 \int^T_0 dt\int _{\mathbb{R}^2}\frac{\partial \varphi}{\partial t}\circ \nabla P(t,\cdot) e(s)\chi_{D_{\rho_t}}ds dz + \int^T_0 L_t[\sigma](\nabla\varphi) dt =0
\end{equation}
 for all $\varphi\in C^1((0,T)\times \mathbb{R}^2).$\\
\end{definition}

\begin{rmk}\label{rmk: continuity quation in terms of V[sigma]}
 In view of lemma \ref{lem: Riesz for V[sigma]}, (\ref{eq: equation of continuity  weak dual sense}) becomes
$$\int^T_0 dt\int _{\mathbb{R}^2}\frac{\partial \varphi}{\partial t} +   \langle \nabla\varphi, V_t[\sigma]\rangle d\sigma_t dt =0$$
 for all $\varphi\in C^1((0,T)\times \mathbb{R}^2).$
That is,
 \begin{equation}\label{eq: continuity quation in terms of V[sigma]}
 \frac{\partial\sigma}{\partial t}+ div (\sigma V_t[\sigma])=0, \qquad  (0,T)\times \mathbb{R}^2
\end{equation}
holds  in the distributional sense.\\

\end{rmk}
\vspace{2pt}
\begin{rmk}
 We want to emphasize that we have defined the velocity field via a Riesz representation when $\sigma$ is  not absolutely continuous with respect to the Lebesgue measure.\\
\end{rmk}

\begin{thm}
Assume  that $\bar F$, $\bar S$ satisfy $(B_1)$ and $(B_2)$, and that $0<  L_0 < R_0$. Let  $\bar \sigma_0\in  \mathcal{P}_{[L_0]}$ and $T>0$  such that $L_0+ C_0T < R_0$. Then, there exists $\sigma_t: [0,T]\longmapsto \left( \mathcal{P}(\mathbb{R}^2),W_1\right) $ $C_0$-Lipschitz continuous  such that $\spt(\sigma_t)\subset [0,R_0]^2$ and

\begin{equation}\label{eq unphysical1:continuity equation 2} 
\begin{cases}
 \dot{\sigma_t}=\chi^{\bar\calH}_{\calH}(\sigma_t)\\
\sigma_{|t=0}=\bar\sigma_0.
\end{cases}
\end{equation}
\end{thm}
\proof{}  We divide the proof into 2 steps.\\
\textbf{Step 1} (Construction of a discrete solution.)\\
We choose $V=(V^1,V^2)$ as provided by Lemma \ref{lem: Riesz for V[sigma]}. For any $\sigma \in \calP_{[R_0]}$, by redefining $V_t[\sigma]$ on a $\sigma$ negligible subset of $\mathbb{R}^2$, we may  assume without loss of generality that $V^1_t[\sigma] , V^2_t[\sigma]\geq 0$ and  $|V_t[\sigma]|\leq C_0$  on $\mathbb{R}^2$ all $t\geq 0$.
Let $N$ be a positive integer. Following a scheme devised by Ambrosio  and Gangbo (see \cite{AmbrosioGangbo}, \cite{Sedjro}), we easily build a sequence of measure-valued curve $t\longmapsto\sigma_{t}^{N}$ satisfying:\\
(a) $t\longmapsto \sigma^N_t$ is Lipschitz continuous  with the Lipschitz constant less than or equal to $C_0$.\\
(b) $W_1(\sigma_t^N,\sigma_0)\leq C_0 T$.\\
(c) $spt(\sigma^N_t)\subset [0, R_0]^2$ for $t\in[0,T]$.\\
(d) $t\longmapsto \sigma^N_t$ satisfies that
$$\begin{cases}
 \frac{\partial\sigma^N}{\partial t}+ \div (\sigma_t^N V_t[\sigma_t^N])=0, \qquad (0,T)\times\mathbb{R}^2 \\
\sigma^{N}_{|t=0}=\bar\sigma_0
\end{cases}
$$

 holds in the distributional sense with 
\begin{equation}\label{eq unphysical1:transport of the momentum for discret solution} 
 V_t[\sigma_t^N]\sigma_t^{N}=\left(\id+(t-[\frac{t}{\tau}]\tau) {V}_{[\frac{t}{\tau}]\tau} [\sigma^{N}_{[\frac{t}{\tau}]\tau}] \right)_{\#}\left( {V}_{[\frac{t}{\tau}]\tau} [\sigma^{N}_{[\frac{t}{\tau}]\tau}]\sigma^{N}_{[\frac{t}{\tau}]\tau}\right).  
\end{equation}

\textbf{Step 2} (Construction of a Lipschitz continuous solution.)\\
In view of (a) and (b), there exists a subsequence  of $\left\lbrace \sigma_t^N\right\rbrace_{N} $ still denoted by $\left\lbrace \sigma_t^N\right\rbrace_{N} $ such that $\left\lbrace \sigma_t^N\right\rbrace_{N} $ converges narrowly to  some $\sigma_t$ for each $t$ fixed independently of $N$. We next show that $\sigma_t$ solves (\ref{eq unphysical1:continuity equation 2}) or equivalently (\ref{eq: continuity quation in terms of V[sigma]}), in view of Remark \ref{rmk: continuity quation in terms of V[sigma]}. For this purpose, we only have to show that, up to some subsequence, $\left\lbrace V_t[\sigma^N_t]\sigma^N_t dt\right\rbrace_{N}$ converges in the sense of distributions to $V_t[\sigma_t]\sigma_tdt$.\\

Let $\phi\in C_c^1((0,T)\times\mathbb{R}^2,\mathbb{R}^2)$. We use (\ref{eq unphysical1:transport of the momentum for discret solution}) to obtain
\begin{equation}\label{eq: control first term in break up measure -vel -time 1} 
 \begin{aligned}
&\left\vert\int_0^T dt \int_{\mathbb{R}^2}\langle \phi(t,x);\;  V_t[\sigma_t^N] \rangle d\sigma_t^N - \int_0^T dt \int_{\mathbb{R}^2}\langle \phi(t,x);\;V_{\tau[\frac{t}{\tau}]}[\sigma_{\tau[\frac{t}{\tau}]}^N] \rangle d\sigma_{\tau[\frac{t}{\tau}]}^N \right\vert\\
&\leq \sum^N_{k=1}\int^{k\tau}_{(k-1)\tau} dt\int_{\mathbb{R}^2} \left\vert \phi\left(t, x+(t-\tau [\frac{t}{\tau}]) V_{\tau[\frac{t}{\tau}]}[\sigma_{\tau[\frac{t}{\tau}]}^N]\right) -\phi(t,x)\right\vert
|V_{\tau[\frac{t}{\tau}]}[\sigma_{\tau[\frac{t}{\tau}]}^N]| d\sigma^N_{\tau [\frac{t}{\tau}]}.\\
 \end{aligned}
\end{equation}
By using  the facts that $\sigma^N_{\tau [\frac{t}{\tau}]}$ has  its support in $[0,R_0]^2$,
$\phi$ is Lipschitz on $[0,T]\times[0,R_0]^2$ and $V$ is bounded, the right hand side of (\ref{eq: control first term in break up measure -vel -time 1}) can be estimated by
\begin{equation}\label{eq: control first term in break up measure -vel -time 2} 
\sum^N_{k=1}\int^{k\tau}_{(k-1)\tau} dt\int_{[0,R_0]^2}  Lip (\phi)\left\vert t-\tau [\frac{t}{\tau}]\right\vert | V_{\tau[\frac{t}{\tau}]}[\sigma_{\tau[\frac{t}{\tau}]}^N]|^2 d\sigma^N_{\tau [\frac{t}{\tau}]}\leq
 C_0^2  Lip(\phi)\frac{T^2}{2N}.
\end{equation}
 We combine (\ref{eq: control first term in break up measure -vel -time 1}) and (\ref{eq: control first term in break up measure -vel -time 2}) to obtain 
\begin{equation}\label{eq: control first term in break up measure -vel -time 3} 
\limsup_{N\rightarrow\infty}\left\vert\int_0^T dt \int_{\mathbb{R}^2}\langle \phi(t,x);\; V_t[\sigma_t^N] \rangle d\sigma_t^N - \int_0^T dt \int_{\mathbb{R}^2}\langle \phi(t,x);\; V_{\tau[\frac{t}{\tau}]}[\sigma_{\tau[\frac{t}{\tau}]}^N]\rangle d\sigma_{\tau[\frac{t}{\tau}]}^N \right\vert =0.
\end{equation}
 By (\ref{eq control variation in time of velocity}),
\begin{equation}\label{eq0: bound on the sequence in second term} 
\left\vert  \int_0^T dt \int_{\mathbb{R}^{2}}\langle \phi(t,x);\; V_{\tau[\frac{t}{\tau}]}[\sigma_{\tau[\frac{t}{\tau}]}^N]  - V_t[\sigma_{\tau[\frac{t}{\tau}]}^N] \rangle d\sigma_{\tau[\frac{t}{\tau}]}^N \right\vert \leq 
\sqrt{R_0}\|\phi\|_{\infty}\int_0^T \sup_{p\in\bar\Delta_{r_0}} \left\vert\mathbb{F}_{\tau\left[ \frac{t}{\tau}\right] }(p)- \mathbb{F}_t(p)\right\vert dt.
\end{equation}

As $|t-\tau\left[\frac{t}{\tau} \right] |\leq \tau=\frac{T}{N}$ and $\mathbb{F}$ is continuous and bounded on $[0,T]\times \Delta_{r_0}$, we use the Lebesgue dominated convergence theorem to obtain that
\begin{equation}\label{eq1: bound on the sequence in second term} 
\begin{aligned}
\limsup_{N\rightarrow\infty}&\left\vert  \int_0^T dt \int_{\mathbb{R}^{2}}\langle \phi(t,x);\; V_{\tau[\frac{t}{\tau}]}[\sigma_{\tau[\frac{t}{\tau}]}^N]  - V_t[\sigma_{\tau[\frac{t}{\tau}]}^N] \rangle d\sigma_{\tau[\frac{t}{\tau}]}^N \right\vert \\
&\leq \sqrt{R_0}\|\phi\|_{\infty}\limsup_{N\rightarrow\infty}\int_0^T \sup_{p\in\Delta_{r_0}} \left\vert\mathbb{F}_{\tau\left[ \frac{t}{\tau}\right] }(p)- \mathbb{F}_t(p)\right\vert dt=0.\\
\end{aligned}
\end{equation}
We note that 
\begin{equation}\label{eq0: bound on the sequence in third term} 
 \left\vert \int_{\mathbb{R}^2}\langle \phi(t,x), V_t[\sigma_{\tau[\frac{t}{\tau}]}^N]\rangle d\sigma_{\tau[\frac{t}{\tau}]}^N \right\vert\leq C_0 \|\phi\|_{\infty}.
\end{equation}

Using (a), we have
$$W_1\left( \sigma_{\tau\left[\frac{t}{\tau}\right]}^N, \sigma_t^N\right)\leq C_0\left\vert t-\tau\left[\frac{t}{\tau} \right] \right\vert  \leq\frac{C_0 T}{N}.$$

 And so,  as $N$ goes to $\infty$, $\left\lbrace \sigma_{\tau\left[\frac{t}{\tau} \right]}^N\right\rbrace_{N} $ converges narrowly to $\sigma_t$ and   lemma \ref{lem: general measure: conv of momentum}  ensures that 
\begin{equation}\label{eq1: convergence on the sequence in third term} 
\int_{\mathbb{R}^2} \langle \phi(t,x);\; V_t[\sigma_{\tau[\frac{t}{\tau}]}^N] \rangle d\sigma_{\tau[\frac{t}{\tau}]}^N \text{ converges a.e to  } \int_{\mathbb{R}^2} \langle \phi(t,x);\; V_t[\sigma_t]\rangle d\sigma_t.
\end{equation}

 We combine (\ref{eq0: bound on the sequence in third term}) and (\ref{eq1: convergence on the sequence in third term}) and use the Lebesgue dominated convergence theorem to obtain that 
\begin{equation}\label{eq2: convergence on the sequence in third term} 
 \int_0^T dt \int_{\mathbb{R}^2}\langle\phi(t,x) V_t[\sigma_{\tau[\frac{t}{\tau}]}^N] \rangle d\sigma_{\tau[\frac{t}{\tau}]}^N  \text{ converges to } \int_0^T dt \int_{\mathbb{R}^2}\langle\phi(t,x), V_t[\sigma_t]\rangle d\sigma_t.
\end{equation}

 In view of (\ref{eq: control first term in break up measure -vel -time 3}) (\ref{eq1: bound on the sequence in second term}) and (\ref{eq2: convergence on the sequence in third term}), we have 
$$\limsup_{N\rightarrow\infty}\left\vert\int_0^T dt \int_{\mathbb{R}^2}\langle\phi(t,x), {V_t[\sigma^N_t]}_t^N\rangle d\sigma_t^N - \int_0^T dt \int_{\mathbb{R}^2}\langle \phi(t,x), V_t[\sigma_t]\rangle d\sigma_t \right\vert =0. $$
As $\phi$ is arbitrary, we obtain that $\left\lbrace V_t[\sigma^N_t]\sigma^N_t dt\right\rbrace_{N}$ converges in the sense of distributions to $V_t[\sigma_t]\sigma_tdt$, which concludes the proof. \endproof

\section*{Acknowledgments}

The Authors would like to thank Prof. Wilfrid Gangbo for his help and comments. They also want to express their gratitude to the anonymous referees who help improve considerably the presentation of this paper. This work  
was mostly completed while MS was a student in the PhD program at the  Georgia Institute of Technology to which he is grateful  for the excellent working conditions. MS would like to gratefully acknowledge RA support from the  
National Science Foundation through grants DMS--0901070 of Prof. W.  Gangbo,  DMS--0807406 of Prof. R. Pan  and DMS--0856485 of Prof. A. Swiech.

\end{document}